\documentclass[10pt,leqno]{article}
\usepackage[T1]{fontenc}
\usepackage[utf8]{inputenc}
\usepackage[letterpaper,textwidth=5.6in,textheight=9in,centering]{geometry}
\usepackage{amsmath,amssymb,amsthm,mathtools,bm,mathrsfs}
\usepackage{booktabs,array,enumitem,microtype,needspace}
\usepackage{ifpdf,aliascnt}
\usepackage[hidelinks]{hyperref}
\usepackage[capitalize,nameinlink]{cleveref}

\numberwithin{equation}{section}
\newtheorem{theorem}{Theorem}[section]
\theoremstyle{plain}
\newaliascnt{lemma}{theorem}
\newtheorem{lemma}[lemma]{Lemma}
\aliascntresetthe{lemma}
\crefname{lemma}{Lemma}{Lemmas}
\theoremstyle{plain}
\newaliascnt{corollary}{theorem}

\aliascntresetthe{corollary}
\crefname{corollary}{Corollary}{Corollaries}
\theoremstyle{plain}
\newaliascnt{proposition}{theorem}
\newtheorem{proposition}[proposition]{Proposition}
\aliascntresetthe{proposition}
\crefname{proposition}{Proposition}{Propositions}
\theoremstyle{plain}
\newaliascnt{definition}{theorem}

\aliascntresetthe{definition}
\crefname{definition}{Definition}{Definitions}
\theoremstyle{remark}
\newaliascnt{assumption}{theorem}
\newtheorem{assumption}[assumption]{Assumption}
\aliascntresetthe{assumption}
\crefname{assumption}{Assumption}{Assumptions}
\theoremstyle{remark}
\newaliascnt{remark}{theorem}
\newtheorem{remark}[remark]{Remark}
\aliascntresetthe{remark}
\crefname{remark}{Remark}{Remarks}
\theoremstyle{remark}
\newaliascnt{example}{theorem}
\newtheorem{example}[example]{Example}
\aliascntresetthe{example}
\crefname{example}{Example}{Examples}

\newcommand{\email}[1]{\href{mailto:#1}{\nolinkurl{#1}}}
\newenvironment{keywords}{\par\small\noindent\textbf{Keywords.}\ }{\par}
\newenvironment{MSCcodes}{\par\small\noindent\textbf{MSC codes.}\ }{\par\medskip}
\date{}

\newcommand{\Trect}{\mathbb T_\pi\times\mathbb T_1}

\newcommand{\supp}{\operatorname{supp}}
\newcommand{\dist}{\operatorname{dist}}

\title{Fourier Energy Strata for the\\
Cubic Schr\"odinger Equation}
\author{Yulin Xie\thanks{State Key Laboratory of Mathematical Sciences,
Academy of Mathematics and Systems Science, Chinese Academy of Sciences,
Beijing 100190, China; and School of Mathematical Sciences,
University of Chinese Academy of Sciences, Beijing 100049, China.
Corresponding author: Yulin Xie (\email{xieyulin@amss.ac.cn}).}}
\ifpdf
\hypersetup{pdftitle={Fourier Energy Strata for the Cubic Schr\"odinger Equation},
pdfauthor={Yulin Xie},
pdfkeywords={nonlinear Schrodinger equation, partial normal form, Fourier energy strata, small divisors}}
\fi
\begin{document}
\maketitle
\begin{abstract}
We prove a geometric hierarchy of Fourier layer energies for the weakly
nonlinear cubic Schr\"odinger equation on $\mathbb T_\pi\times\mathbb T_1$,
retaining its nontrivial resonant dynamics.
For initial Fourier support in a rectangle, form nested rectangles with
half-widths $(2j+1)$ times the initial ones; layers are the differences
between consecutive rectangles.
On $0\le t\le T\delta^{-1}$, the energy in layer $j$, measured
by its squared $H^1$ norm, is at most $C_N\delta^{2j}$ for
$1\le j\le N$; the energy outside the $N$th rectangle is at most
$C_N\delta^{2N+2}$.  These bounds hold at every fixed depth $N$
under a uniform $H^s$ bound, $s>2$, with constants depending on $N$.
The Sobolev bound holds for every fixed $T$ in the defocusing case.
We also obtain layer estimates for support bounded in one coordinate
and for geometrically decaying initial Fourier tails.
The proof uses a filtered Poincar\'e--Dulac normal form: input layers
select which nonresonant interactions are removed, while the same
layer weights control the interactions left in the equation.
Finite coordinate changes and stability estimates transfer this
hierarchy to the exact solution.
\end{abstract}

\section{Introduction}\label{sec:introduction}
We prove a geometric hierarchy of Fourier layer energies for the weakly
nonlinear cubic Schr\"odinger equation
\[
 i\partial_tu^\delta=-\Delta u^\delta+
 \lambda\delta|u^\delta|^2u^\delta,
 \qquad u^\delta(0)=u_0,\qquad \lambda\in\{-1,+1\},
\]
on $\mathbb T_\pi\times\mathbb T_1$, where
$\mathbb T_L=\mathbb R/(L\mathbb Z)$ and $0<\delta\le1$.
Cubic Schr\"odinger equations arise as slowly varying envelope models
in nonlinear optics, where the cubic term represents intensity-dependent
refraction \cite{Kelley1965}, and in water-wave theory
\cite{Zakharov1968}.  Their periodic versions provide a setting for
studying how nonlinear interactions redistribute energy among Fourier
modes \cite{FaouGermainHani2016}.  Here $\delta$ measures the interaction
strength, and the fixed rectangular torus has nontrivial cubic resonances.
For initial Fourier support
in a rectangle, enlarge its half-widths by the factors $2j+1$ and define
layers as the differences between consecutive rectangles.
Under a uniform exact-solution $H^s$ bound, $s>2$, we show that, for every
fixed depth $N$ and sufficiently small $\delta$, the squared $H^1$ norm
in layer $j$ is at most $C_N\delta^{2j}$ for $1\le j\le N$, and the
energy outside the $N$th rectangle is at most $C_N\delta^{2N+2}$,
uniformly on $0\le t\le T\delta^{-1}$ for fixed $T>0$.
The resonant dynamics are retained throughout the analysis.

Although the nonlinear coefficient is small, its cumulative effect
over times of order $\delta^{-1}$ can be of order one.
Resonant energy exchange between modes cannot be discarded as a
small correction.  Yet the desired bound must gain a further factor
$\delta^2$ at each successive layer, through any prescribed finite
depth.

The proof uses a \emph{filtered} Poincar\'e--Dulac normal form.
Unlike a full normal-form construction, it removes only nonresonant
interactions selected by their input Fourier layers at each order.
A layered approximate solution and stability estimates connect the
transformed dynamics to the exact solution.

Beyond finite rectangular support, Theorem~\ref{cor:fixed-order} covers
support bounded in only one coordinate and geometrically decaying
initial layer energies.  In the latter case, the Fourier support may
be unbounded in both coordinates.  The regularity requirement is
$s>2$ in all three cases.  For the defocusing equation, the uniform-in-$\delta$
$H^s$ solution bound need not be assumed separately.
The period $\pi$ can be replaced by any fixed $L>0$ with $L^2\notin\mathbb Q$.

The closest predecessor is Staffilani and
Wilson~\cite{StaffilaniWilson2020}, who study the same potential-free
defocusing cubic Schr\"odinger equation on irrational rectangular tori.
For $\lambda=+1$, the rescaling $v=\sqrt{\delta}\,u^\delta$ yields
their small-data formulation, up to the sign convention.
Amplitude and time rescalings relate the small-data long-time problem,
the weakly nonlinear long-time problem, and a highly oscillatory problem
on a fixed time interval; Liu, Wang, and Zhao~\cite{LiuWangZhao2025}
use these relations and a two-scale formulation to construct numerical
integrators on equal-sided periodic domains.
We use the weakly nonlinear formulation to organize the normal form
order by order and track the powers of $\delta$ in the Fourier energy
hierarchy.
Staffilani and Wilson use a partial normal-form reduction and the
coordinatewise structure of cubic resonances to control leakage from
an initial Fourier box;
their proof gives Sobolev comparison estimates cubic in the initial
amplitude.  Their argument retains nontrivial resonances and some
nonresonant interactions involving high modes.

Normal forms \cite{Bambusi2003,BambusiGrebert2006,Grebert2007}
and Galerkin averaging \cite{Bambusi2005} provide a broader framework
for long-time stability of Hamiltonian PDEs.
For NLS with a nonresonant convolution potential, Gauckler and
Lubich~\cite{GaucklerLubich2010} prove near-conservation of Fourier
actions over arbitrary polynomial times, while
Gauckler~\cite[Corollary~2.16]{Gauckler2010} obtains weighted energy
strata in odd-multiple bands, for small, sufficiently regular data.
Gauckler, Hairer, Lubich, and Weiss~\cite{GaucklerHairerLubichWeiss2012}
prove metastable energy strata for weakly nonlinear wave equations
by modulated Fourier expansions, excluding finitely many mass parameters
at fixed depth in a bounded mass interval.
Faou, Gauckler, and Lubich~\cite{FaouGaucklerLubich2013} prove
plane-wave orbital stability over arbitrary polynomial times for small,
regular perturbations and a full-measure set of linearly stable mass
parameters satisfying reduced nonresonance conditions.
Our estimates require neither such a choice of mass nor proximity
to a single plane wave.

For resonant NLS, Colliander et al.~\cite{CKSTT2010} construct large
finite growth of high Sobolev norms on the standard two-dimensional
torus, while Carles and Faou~\cite{CarlesFaou2012} describe high-mode
creation from five-mode data through resonant amplitude equations.
Quasi-resonant Sobolev growth on irrational
tori \cite{GiulianiGuardia2022} cautions against extending this
localization to all longer time scales.

Our contribution is to refine the exterior leakage control of
Staffilani and Wilson into geometric layer estimates through every
fixed finite depth, together with a bound on the total exterior tail.
In contrast to the related nonresonant energy-strata results, we retain
and control nontrivial resonances.
Overall localization or total-energy control does not give the
increasingly sharp leakage bounds required in successive layers.
We overcome this difficulty by filtering nonresonant terms according
to their input Fourier layers at each order: the selected terms admit
controlled coordinate changes, while the same layer weights bound
the retained terms on the layered approximation.

Section~\ref{sec:main-results} introduces the problem and assumptions
and states the main theorem.
Section~\ref{sec:normal-form} constructs the filtered normal form,
starting with one explicit calculation, and Section~\ref{sec:transfer}
proves the estimates for the exact solution.
Section~\ref{sec:long-time} examines the limits of the result through
a quintic example and a stability calculation.
Appendix~\ref{app:envelope} proves the defocusing envelope, and
Appendix~\ref{app:convolution-phase} supplies the convolution and phase estimates.

\Needspace{8\baselineskip}
\section{Problem and main theorem}
\label{sec:main-results}
\subsection{Equation and Fourier notation}
Let $0<\delta\le1$ and fix
$\lambda\in\{-1,+1\}$.  We consider
\begin{equation}\label{eq:nls-ivp}
 i\partial_tu^\delta=-\Delta u^\delta
 +\lambda\delta|u^\delta|^2u^\delta,
 \qquad u^\delta(0)=u_0\in H^s(\Trect),\quad s>2,
\end{equation}
where $u_0$ is independent of $\delta$.  Here
$\mathbb T_\pi=\mathbb R/(\pi\mathbb Z)$ and
$\mathbb T_1=\mathbb R/\mathbb Z$.  Thus the solution and the initial datum
are $\pi$-periodic in $x$ and $1$-periodic in $y$:
\[
u^\delta(t,x+\pi,y)=u^\delta(t,x,y),\qquad
u^\delta(t,x,y+1)=u^\delta(t,x,y).
\]
The sign $+1$ is defocusing and $-1$ is focusing.  With normalized Haar
measure $\pi^{-1}\,dx\,dy$, set
\[
 e_{m,n}=e^{2imx+2\pi iny},\quad
 \omega_{m,n}=4m^2+4\pi^2n^2,\quad
 \langle k\rangle=(1+\omega_k)^{1/2}.
\]
For a Fourier sequence $z$, let
$\|z\|_{h^r}^2=\sum_k\langle k\rangle^{2r}|z_k|^2$.
The Fourier transform identifies $H^r$ with $h^r$.
For $S\subset\mathbb Z^2$, let $P_S$ be the Fourier projection onto
the modes in $S$, defined by
\[
 \widehat{P_Su}_k=\mathbf 1_S(k)\widehat u_k.
\]
We use the same notation on Fourier sequences: $(P_Sz)_k=\mathbf 1_S(k)z_k$.
The main interval is $0\le t\le T\delta^{-1}$, with fixed $T>0$.

The choice of $\pi$ is not essential.  The main theorem and its proof
apply on $\mathbb T_L\times\mathbb T_1$ for every fixed $L>0$ with
$L^2\notin\mathbb Q$, with constants depending on $L$.
The spectrum then becomes $\omega_{m,n}=4\pi^2(m^2/L^2+n^2)$,
and interaction phases have the form
$(4\pi^2/L^2)(A+L^2B)$ with integers $A,B$.
Irrationality gives the coordinatewise resonance identities and the
filtered phase gap used below.
We keep $L=\pi$ in the displayed calculations.

\Needspace{5\baselineskip}
\subsection{Solution envelope and Fourier layers}
\begin{assumption}[Exact-solution envelope]\label{ass:Hs-envelope}
There are $\bar\delta\in(0,1]$ and $R_s<\infty$ such that the
solution of \eqref{eq:nls-ivp} exists on $[0,T\delta^{-1}]$ and
\[
 \sup_{0<\delta\le\bar\delta}
 \sup_{0\le t\le T\delta^{-1}}\|u^\delta(t)\|_{H^s}\le R_s.
\]
\end{assumption}
\begin{lemma}[Defocusing envelope]\label{lem:defocusing-envelope}
If $\lambda=+1$, $s>2$, and $u_0\in H^s$, then
Assumption~\ref{ass:Hs-envelope} holds for every fixed $T>0$ with
$\bar\delta=1$.  The constant may depend on $T,s,u_0$.
\end{lemma}
This is the periodic $H^s$ version of the Br\'ezis--Gallou\"et argument
\cite{BrezisGallouet1980}, with $\delta$ retained in the
estimates; Appendix~\ref{app:envelope} gives the proof and its parameter
dependence.  For the focusing equation we retain
Assumption~\ref{ass:Hs-envelope}; local $H^s$ theory supplies it for
sufficiently small $T$ depending on $\|u_0\|_{H^s}$
\cite[Section~3.3]{Tao2006}.

Fix $M_0,N_0\ge1$.  We use one notation for the Fourier regions in
the three support geometries:
\begin{equation}\label{eq:box-bands}
 B_j=
 \begin{cases}
 \{(m,n):|m|\le(2j+1)M_0,\ |n|\le(2j+1)N_0\},
       &\text{rectangle},\\
 \mathbb Z\times\{n:|n|\le(2j+1)N_0\},
       &\text{strip bounded in }n,\\
 \{m:|m|\le(2j+1)M_0\}\times\mathbb Z,
       &\text{strip bounded in }m.
 \end{cases}
\end{equation}
After choosing one geometry, keep that choice throughout the
construction and put
\[
 B_{-1}=\varnothing,\qquad \Gamma_j=B_j\setminus B_{j-1}.
\]
The geometric initial-tail condition below uses the rectangular choice.
The corresponding layer energy and exterior energy are
\[
 \mathscr E_j^\delta(t)
 =\sum_{k\in\Gamma_j}(1+\omega_k)|\widehat u_k^\delta(t)|^2,
 \qquad
 \mathscr E_{>N}^\delta(t)
 =\sum_{k\notin B_N}(1+\omega_k)|\widehat u_k^\delta(t)|^2.
\]
These are the squared $H^1$ norms of the corresponding Fourier projections.

\subsection{Main theorem}
\begin{theorem}[Fixed-order Fourier strata]\label{cor:fixed-order}
Assume Assumption~\ref{ass:Hs-envelope} and fix an integer $N\ge1$.
For fixed $T$, Fourier regions, and the prescribed $H^s$ envelope,
there are $C_N<\infty$ and $\delta_N^*>0$, independent of
$\delta$, such that the following hold with
$\delta_N=\min\{\bar\delta,\delta_N^*\}$.
All time suprema are over $0\le t\le T\delta^{-1}$ and
$0<\delta\le\delta_N$.
\begin{enumerate}[label=\textup{(\roman*)},leftmargin=2.4em]
\item For the rectangular choice of $B_j$, if $\supp\widehat u_0\subset B_0$, then
\begin{equation}\label{eq:main-box-layer}
 \sup_t\|P_{\Gamma_j}u^\delta(t)\|_{H^1}
 \le C_N\delta^j\quad(1\le j\le N),\qquad
 \sup_t\|P_{B_N^c}u^\delta(t)\|_{H^1}
 \le C_N\delta^{N+1}.
\end{equation}
\item For either strip choice of $B_j$, if
$\supp\widehat u_0\subset B_0$, the same estimates
\eqref{eq:main-box-layer} hold with that choice of $B_j,\Gamma_j$.
\item For the rectangular choice of $B_j$, if $E_*>0$, $0<\eta<1$, and
\begin{equation}\label{eq:geometric-initial-tail}
 \|P_{\Gamma_j}u_0\|_{H^1}^2\le E_*\eta^j\qquad(j\ge0),
\end{equation}
then there is $C_{N,E_*}$, also allowed to depend on $E_*$ but
independent of $0<\eta<1$, such that
\begin{align}
 \sup_t\|P_{\Gamma_j}u^\delta(t)\|_{H^1}
 &\le C_{N,E_*}
       \bigl(E_*^{1/2}\eta^{j/2}+\delta^j\bigr),
       &&1\le j\le N,\label{eq:main-geometric-layer}\\
 \sup_t\|P_{B_N^c}u^\delta(t)\|_{H^1}
 &\le C_{N,E_*}
 \left(\frac{E_*^{1/2}\eta^{(N+1)/2}}{\sqrt{1-\eta}}
       +\delta^{N+1}\right).
 \label{eq:main-geometric-tail}
\end{align}
\end{enumerate}
The constants can be chosen for both signs of $\lambda$.
\end{theorem}
For the defocusing equation, Lemma~\ref{lem:defocusing-envelope}
supplies the envelope for every fixed $T$, so no additional
solution-envelope assumption is needed.  Here $\bar\delta=1$;
the normal-form threshold $0<\delta\le\delta_N$ remains.
In particular, parts~\textup{(i)}--\textup{(ii)} give the geometric
energy hierarchy
\[
 \sup_t\mathscr E_j^\delta(t)\le C_N^2\delta^{2j},
 \qquad
 \sup_t\mathscr E_{>N}^\delta(t)\le C_N^2\delta^{2N+2}.
\]
Part~\textup{(iii)} gives
$\sup_t\mathscr E_j^\delta(t)
 \le2C_{N,E_*}^2(E_*\eta^j+\delta^{2j})$.
The constants need not be uniform in $N$ or the aspect ratio.
Uniformity in $\eta$ uses the common $H^s$ envelope; the factor
$(1-\eta)^{-1/2}$ in \eqref{eq:main-geometric-tail} comes from summing
the initial exterior energies.

\section{Filtered normal forms}\label{sec:normal-form}
\subsection{First-order filtered normal form}
\label{subsec:one-step}
We calculate the first change of variables in full, including the terms
that the Fourier filter leaves in the equation.  The calculation uses
the same notation for the three choices of \(B_j\) in
\eqref{eq:box-bands}.

\smallskip\noindent\emph{Interaction variables.}
Use physical time $t$ throughout.
Define the interaction coefficients by
\[
 a_k(t)=e^{i\omega_kt}\widehat u_k^\delta(t).
\]
The exponential removes the linear Schr\"odinger oscillation.
Differentiating and using \eqref{eq:nls-ivp} gives
\begin{equation}\label{eq:slow-interaction}
 \frac{d a_k}{dt}
 =-i\lambda\delta\sum_{k_1-k_2+k_3=k}
   e^{-i\Omega t}
   a_{k_1}\overline{a_{k_2}}a_{k_3},
 \qquad
 \Omega=\omega_{k_1}-\omega_{k_2}+\omega_{k_3}-\omega_k.
\end{equation}
When needed, write the phase as $\Omega(k_1,k_2,k_3,k)$.
Abbreviate the cubic sum without its prefactor $\delta$ by
\[
 F_k(t,z)
 =-i\lambda\sum_{k_1-k_2+k_3=k}
   e^{-i\Omega t}z_{k_1}\overline{z_{k_2}}z_{k_3}.
\]
Then $d a/dt=\delta F(t,a(t))$.

We use \(\mathcal R\) to select the resonant terms, for which
\(\Omega=0\), and \(\mathcal N\) to select the nonresonant terms,
for which \(\Omega\ne0\).  Explicitly,
\begin{align}
 (\mathcal RF)_k(z)
 &=-i\lambda
   \sum_{\substack{k_1-k_2+k_3=k\\\Omega=0}}
        z_{k_1}\overline{z_{k_2}}z_{k_3},
 \label{eq:first-resonant-sum}\\
 (\mathcal NF)_k(t,z)
 &=-i\lambda
   \sum_{\substack{k_1-k_2+k_3=k\\\Omega\ne0}}
        e^{-i\Omega t}z_{k_1}\overline{z_{k_2}}z_{k_3}.
 \label{eq:first-nonresonant-sum}
\end{align}
Their sum is \(F\).  We leave the first sum unchanged and split the
second sum further according to the input layers.

\smallskip\noindent\emph{Layer projections and interaction filtering.}
Fix the terminal depth \(K=N+1\).  For the chosen family \(B_j\), define
\begin{equation}\label{eq:saturated-cost}
 P_j=P_{\Gamma_j}\quad(0\le j<K),\qquad
 P_K=P_{B_{K-1}^{\,c}},\qquad
 \sum_{j=0}^K P_j=I.
\end{equation}
Recall from \eqref{eq:box-bands} that, for rectangles, \(B_{K-1}\)
is the set of modes with \(|m|\le(2K-1)M_0\) and
\(|n|\le(2K-1)N_0\); for a strip, only the corresponding coordinate
bound is imposed.  It is the union of the layers
\(\Gamma_0,\ldots,\Gamma_{K-1}\).
For \(j<K\), \(P_j\) projects onto the \(j\)th Fourier layer.
The projection \(P_K\), called the \emph{outer block}, collects all
frequencies outside \(B_{K-1}\).  It is the entire exterior, rather
than only the next annular layer \(\Gamma_K\).
For an existing vector \(z\), let
\[
 z_{k,j}:=(P_jz)_k,\qquad z_k=\sum_{j=0}^K z_{k,j}.
\]
The last equality decomposes \(z\); it introduces no new unknown.
Substitution in \eqref{eq:first-nonresonant-sum} gives
\[
 (\mathcal NF)_k(t,z)
 =-i\lambda
  \sum_{\substack{k_1-k_2+k_3=k\\\Omega\ne0}}
  e^{-i\Omega t}
  \sum_{j_1,j_2,j_3=0}^K
  z_{k_1,j_1}\overline{z_{k_2,j_2}}z_{k_3,j_3}.
\]

We select for elimination the input triples in
\[
 \mathcal J_{K,1}
 =\{(j_1,j_2,j_3)\in\{0,\ldots,K\}^3:
                         1+j_1+j_2+j_3\le K\}.
\]
Let \(\mathcal Q_{K,1}\mathcal NF\) denote the preceding sum restricted
to \(\mathcal J_{K,1}\).  We call a nonresonant term \emph{retained}
if it is left in the transformed equation at the order being treated,
and \emph{unretained} if the coordinate correction cancels it at that
order.  Thus \(\mathcal Q_{K,1}\mathcal NF\) is the unretained part;
the complementary sum is retained and will be controlled by the layer
estimates in Section~\ref{sec:transfer}.  The resonant sum is also retained.

\smallskip\noindent\emph{Splitting the nonresonant sum and choosing the correction.}
The retained nonresonant part is the following explicit sum:
\begin{equation}\label{eq:first-unretained-sum}
 (E_{0,K})_k(t,z)
 =-i\lambda
 \sum_{\substack{k_1-k_2+k_3=k\\\Omega\ne0}}
 e^{-i\Omega t}
 \sum_{\substack{0\le j_1,j_2,j_3\le K\\1+j_1+j_2+j_3>K}}
 z_{k_1,j_1}\overline{z_{k_2,j_2}}z_{k_3,j_3}.
\end{equation}
Thus
\[
 F=\mathcal RF+\mathcal Q_{K,1}\mathcal NF+E_{0,K}.
\]
The two parts that this first cancellation will leave together are
\(\mathcal RF\) and \(E_{0,K}\).  We denote their sum by
\[
 C_0(t,z):=\mathcal RF(z)+E_{0,K}(t,z).
\]
In particular, \(C_0\) is not asserted to be entirely resonant.

For the unretained nonresonant sum, integrate the exponential while
holding the input vector fixed.  Since
\(\partial_t(e^{-i\Omega t}/(-i\Omega))
=e^{-i\Omega t}\), choose
\begin{equation}\label{eq:first-generator}
 (G_1)_k(t,z)
 =\lambda
 \sum_{\substack{k_1-k_2+k_3=k\\\Omega\ne0}}
 \frac{e^{-i\Omega t}}{\Omega}
 \sum_{\boldsymbol j\in\mathcal J_{K,1}}
 z_{k_1,j_1}\overline{z_{k_2,j_2}}z_{k_3,j_3}.
\end{equation}
This \(G_1\) defines the desired correction \(\delta G_1\) to the
variables.  It satisfies
\[
 \partial_t G_1=\mathcal Q_{K,1}\mathcal NF,\qquad
 F-\partial_t G_1=C_0.
\]
Thus its partial time derivative gives the unretained nonresonant sum to be cancelled.
The condition \(\Omega\ne0\) makes each division well-defined.
For a strip, however, the sum still involves infinitely many modes.
To bound \(G_1\) as a cubic map on \(h^s\), we also need a uniform
bound on \(1/|\Omega|\) over the terms selected for elimination and a convolution
estimate.  Lemmas~\ref{lem:majorant-calculus} and~\ref{prop:retained-gap}
give these bounds, with proofs in Appendix~\ref{app:convolution-phase}.

\smallskip\noindent\emph{Differentiating the change of coordinates.}
Define the new variable $b(t)$ by
\begin{equation}\label{eq:first-additive-map}
 a(t)=b(t)+\delta G_1\bigl(t,b(t)\bigr).
\end{equation}
Theorem~\ref{prop:partial-normal-form} establishes invertibility on suitable domains.
Both arguments of $G_1$ vary along $b(t)$, whereas $\partial_tG_1$
holds the state fixed.  With $\delta$ fixed, the chain rule gives
\begin{align}
 \frac{d}{dt}G_1\bigl(t,b(t)\bigr)
 &=\partial_t G_1(t,b)
   +D_bG_1(t,b)\left[\frac{db}{dt}\right],
 \label{eq:first-chain-rule}\\
 \frac{da}{dt}
 &=\frac{db}{dt}+\delta\partial_t G_1(t,b)
   +\delta D_bG_1(t,b)\left[\frac{db}{dt}\right].
 \nonumber
\end{align}
Substitute $d a/dt=\delta F(t,a)$ and $a=b+\delta G_1(t,b)$,
then move $\delta\partial_tG_1$ to the right:
\begin{equation}\label{eq:first-implicit-equation}
 (I+\delta D_bG_1)\frac{db}{dt}
 =\delta F(t,b+\delta G_1)-\delta\partial_t G_1.
\end{equation}
Here $G_1$ and its derivatives are evaluated at $(t,b(t))$.

To see what the state derivative on the left means, put
\(h_{k,j}=(P_jh)_k\).  Differentiating the three factors of
\eqref{eq:first-generator} separately gives
\begin{equation}\label{eq:first-generator-derivative}
 \begin{split}
 (D_bG_1[h])_k
 ={}&\lambda
 \sum_{\substack{k_1-k_2+k_3=k\\\Omega\ne0}}
 \frac{e^{-i\Omega t}}{\Omega}
 \sum_{\boldsymbol j\in\mathcal J_{K,1}}
 \bigl(
 h_{k_1,j_1}\overline{b_{k_2,j_2}}b_{k_3,j_3}\\
 &\qquad{}
 +b_{k_1,j_1}\overline{h_{k_2,j_2}}b_{k_3,j_3}
 +b_{k_1,j_1}\overline{b_{k_2,j_2}}h_{k_3,j_3}
 \bigr).
 \end{split}
\end{equation}
The derivative is over the underlying real space; the conjugated slot
therefore contributes \(\overline{h_{k_2,j_2}}\).

\smallskip\noindent\emph{Expanding the transformed nonlinearity.}
Write \(g=G_1(t,b)\), with components given by
\eqref{eq:first-generator}.  Since the nonlinearity is cubic, its
Taylor expansion in the real parameter \(\delta\), at fixed
\(t,b\), terminates exactly:
\begin{equation}\label{eq:first-cubic-expansion}
 \begin{split}
 F_k(t,b+\delta g)
 ={}&-i\lambda\sum_{k_1-k_2+k_3=k}e^{-i\Omega t}
 \Bigl[
 b_{k_1}\overline{b_{k_2}}b_{k_3}\\
 &+\delta\bigl(
 g_{k_1}\overline{b_{k_2}}b_{k_3}
 +b_{k_1}\overline{g_{k_2}}b_{k_3}
 +b_{k_1}\overline{b_{k_2}}g_{k_3}\bigr)\\
 &+\delta^2\bigl(
 g_{k_1}\overline{g_{k_2}}b_{k_3}
 +g_{k_1}\overline{b_{k_2}}g_{k_3}
 +b_{k_1}\overline{g_{k_2}}g_{k_3}\bigr)\\
 &+\delta^3g_{k_1}\overline{g_{k_2}}g_{k_3}
 \Bigr].
 \end{split}
\end{equation}
The three terms multiplying \(\delta\) arise by replacing exactly one
input by \(g\); those multiplying \(\delta^2\) replace exactly two;
the last term replaces all three.  We now abbreviate these already
displayed sums by
\begin{align}
 (V_1)_k
 &:=-i\lambda\sum_{k_1-k_2+k_3=k}e^{-i\Omega t}
 \bigl(g_{k_1}\overline{b_{k_2}}b_{k_3}
 +b_{k_1}\overline{g_{k_2}}b_{k_3}
 +b_{k_1}\overline{b_{k_2}}g_{k_3}\bigr),
 \label{eq:first-linear-coefficient}\\
 (W_1)_k
 &:=-i\lambda\sum_{k_1-k_2+k_3=k}e^{-i\Omega t}
 \Bigl(
 g_{k_1}\overline{g_{k_2}}b_{k_3}
 +g_{k_1}\overline{b_{k_2}}g_{k_3}
 +b_{k_1}\overline{g_{k_2}}g_{k_3}
 +\delta g_{k_1}\overline{g_{k_2}}g_{k_3}
 \Bigr).
 \label{eq:first-higher-coefficient}
\end{align}
Thus \(V_1=D_bF[g]\) is independent of \(\delta\) at fixed
\(t,b\), whereas \(W_1=W_1(t,b;\delta)\) includes the
remaining two orders.  Equation~\eqref{eq:first-cubic-expansion} reads
\[
 F(t,b+\delta g)
 =F(t,b)+\delta V_1+\delta^2W_1.
\]
These coefficients are concrete Fourier sums.  For example, inserting
\eqref{eq:first-generator} into the first term of \(V_1\) gives
\begin{equation}\label{eq:first-quintic-example}
 \begin{split}
 -i\lambda^2
 \sum_{\substack{k_1-k_2+k_3=k\\
                  \ell_1-\ell_2+\ell_3=k_1\\\Omega_{\rm in}\ne0}}
 \frac{e^{-i(\Omega_{\rm out}+\Omega_{\rm in})t}}
      {\Omega_{\rm in}}
 \sum_{\boldsymbol j\in\mathcal J_{K,1}}
 b_{\ell_1,j_1}\overline{b_{\ell_2,j_2}}b_{\ell_3,j_3}
 \overline{b_{k_2}}b_{k_3},
 \end{split}
\end{equation}
where \(\Omega_{\rm out}=\Omega(k_1,k_2,k_3,k)\) and
\(\Omega_{\rm in}=\Omega(\ell_1,\ell_2,\ell_3,k_1)\).
The internal frequency \(\omega_{k_1}\) cancels, leaving the total phase
\[
 \Omega_{\rm out}+\Omega_{\rm in}
 =\omega_{\ell_1}-\omega_{\ell_2}+\omega_{\ell_3}
   -\omega_{k_2}+\omega_{k_3}-\omega_k.
\]
For the term with \(\overline{g_{k_2}}\), the inner indices instead
satisfy \(\ell_1-\ell_2+\ell_3=k_2\).  With
\(\Omega_{\rm in}^{(2)}=\Omega(\ell_1,\ell_2,\ell_3,k_2)\),
complex conjugation reverses its phase, giving
\(\Omega_{\rm out}-\Omega_{\rm in}^{(2)}\).
Thus the new coefficient contains quintic products and new total
phases, which can be zero or nonzero.

\smallskip\noindent\emph{Solving explicitly for the new differential equation.}
Using \(F-\partial_t G_1=C_0\) in
\eqref{eq:first-implicit-equation} now gives
\[
 (I+\delta D_bG_1)\frac{db}{dt}
 =\delta C_0+\delta^2 V_1+\delta^3W_1.
\]
Let \(A_1=I+\delta D_bG_1\) record the operator on the left.
On the domain where \(A_1\) is invertible, apply \(A_1^{-1}\) to
both sides:
\[
 \frac{db}{dt}
 =\delta A_1^{-1}C_0+\delta^2 A_1^{-1}V_1+\delta^3A_1^{-1}W_1.
\]
To separate the powers of \(\delta\), use the exact identity
\[
 A_1^{-1}(I+\delta D_bG_1)=I,
 \qquad
 A_1^{-1}=I-\delta A_1^{-1}D_bG_1.
\]
First apply this identity to \(A_1^{-1}C_0\), and combine the two
terms carrying a factor \(\delta^2\):
\[
 \begin{aligned}
 \frac{db}{dt}
 &=\delta C_0-\delta^2 A_1^{-1}D_bG_1[C_0]
        +\delta^2 A_1^{-1}V_1+\delta^3A_1^{-1}W_1\\
 &=\delta C_0+\delta^2 A_1^{-1}\bigl(V_1-D_bG_1[C_0]\bigr)
        +\delta^3A_1^{-1}W_1.
 \end{aligned}
\]
Apply the same identity once more, now to
\(A_1^{-1}(V_1-D_bG_1[C_0])\).  This gives
\[
 \begin{aligned}
 \frac{db}{dt}
 ={}&\delta C_0+\delta^2\bigl(V_1-D_bG_1[C_0]\bigr)\\
 &+\delta^3A_1^{-1}
 \Bigl(W_1-D_bG_1\bigl[V_1-D_bG_1[C_0]\bigr]\Bigr).
 \end{aligned}
\]
The quintic coefficient has now been obtained explicitly; denote it by
\begin{equation}\label{eq:first-quintic-coefficient}
 H_1:=V_1-D_bG_1[C_0].
\end{equation}
Both terms are explicit: \(V_1\) is the sum in
\eqref{eq:first-linear-coefficient}, and \(D_bG_1[C_0]\) is
\eqref{eq:first-generator-derivative} with \(h=C_0\), whose components
are given by \eqref{eq:first-resonant-sum} and
\eqref{eq:first-unretained-sum}.
Likewise, denote the remaining coefficient by
\[
 \mathcal E_{1,\delta}
 :=A_1^{-1}\bigl(W_1-D_bG_1[H_1]\bigr).
\]
All identities above are exact, so this definition includes the entire
remainder.
Consequently the first transformed equation, written componentwise, is
\begin{equation}\label{eq:one-step-result}
 \begin{split}
 \frac{db_k}{dt}
 ={}&-i\lambda\delta
       \sum_{\substack{k_1-k_2+k_3=k\\\Omega=0}}
       b_{k_1}\overline{b_{k_2}}b_{k_3}\\
 &-i\lambda\delta
       \sum_{\substack{k_1-k_2+k_3=k\\\Omega\ne0}}
       e^{-i\Omega t}
       \sum_{\substack{0\le j_1,j_2,j_3\le K\\1+j_1+j_2+j_3>K}}
       b_{k_1,j_1}\overline{b_{k_2,j_2}}b_{k_3,j_3}\\
 &+\delta^2(H_1)_k(t,b)
       +\delta^3(\mathcal E_{1,\delta})_k(t,b).
 \end{split}
\end{equation}
The first line is the original cubic resonant sum.  The second is the
retained nonresonant sum, left intact rather than divided by its phase.
The third contains the explicitly computed quintic correction and the
exact higher-order remainder.  On the domains established in
Theorem~\ref{prop:partial-normal-form}, this last remainder is uniformly
bounded for small \(\delta\).  The retained nonresonant sum is controlled
differently: its input indices satisfy \(j_1+j_2+j_3\ge K\), a layer
condition whose contribution is estimated in Section~\ref{sec:transfer}.
Equation~\eqref{eq:one-step-result} completes the first change of
variables.  It displays exactly what has been removed and what remains,
without treating $E_{0,K}$ as uniformly $O(\delta^K)$
on an arbitrary Sobolev ball.

\subsection{Recursion of the transformed equation}
\label{subsec:recursive-construction}
\label{sec:partial-normal-form}
We derive one complete step, including the remainder passed to the
next step.  Write $v^{[0]}=a$ and factor out the common $\delta$
from the right-hand side.  After $q-1$ coordinate changes, write
\begin{equation}\label{eq:current-transformed-equation}
 \partial_t v^{[q-1]}
 =\delta X^{[q-1]}_\delta(t,v^{[q-1]}),
 \qquad X^{[0]}_\delta=F.
\end{equation}
The first calculation gives
$X^{[1]}_\delta=C_0+\delta H_1+\delta^2\mathcal E_{1,\delta}$.
Write $Z_0=\mathcal RF$, so that $C_0=Z_0+E_{0,K}$; at later orders,
$Z_r$ and $E_{r,K}$ will denote the resonant and retained nonresonant
parts.  Step $q$, $1\le q\le K$, acts on the coefficient of
$\delta^{q-1}$ while preserving those of lower order.

Assume the preceding field is smooth in $\delta$ near zero, as
verified on common domains in Section~\ref{subsec:fixed-depth-construction}.
Its Taylor expansion at fixed $(t,z)$ is
\begin{equation}\label{eq:step-taylor}
 X^{[q-1]}_\delta(t,z)
 =\sum_{r=0}^{q-2}\delta^r\bigl(Z_r(z)+E_{r,K}(t,z)\bigr)
  +\delta^{q-1}H_{q-1}(t,z)
  +\delta^q\mathcal T_{q-1}(t,z;\delta).
\end{equation}
The sum contains the coefficients fixed by earlier steps and is empty
for $q=1$.  Taylor's formula defines the next coefficient and exact
remainder by
\begin{equation}\label{eq:coefficient-definition}
 H_{q-1}(t,z)=\frac1{(q-1)!}
 \left.\partial_\rho^{q-1}X^{[q-1]}_\rho(t,z)\right|_{\rho=0},
\end{equation}
\begin{equation}\label{eq:step-input-remainder}
 \mathcal T_{q-1}(t,z;\delta)
 =\frac1{(q-1)!}\int_0^1(1-v)^{q-1}
 \left.\partial_\rho^qX^{[q-1]}_\rho(t,z)\right|_{\rho=v\delta}
 \,\mathrm dv.
\end{equation}
These formulas include $\delta=0$.  Initially $H_0=F$ and
$\mathcal T_0=0$; for $q=2$ they recover $H_1$ and
$\mathcal T_1=\mathcal E_{1,\delta}$ from Section~\ref{subsec:one-step}.

To construct the correction from $H_{q-1}$, we use the Fourier form
already visible in the cubic and quintic calculations.  A cubic
monomial has one more unconjugated factor than conjugated factors.
Replacing a $z$ slot by a monomial with this property, or a $\bar z$
slot by its conjugate, preserves the difference.  Consequently a
degree-$2q+1$ monomial has $q+1$ factors $z$ and $q$ factors $\bar z$.
We can relabel these commuting factors to alternate the two types.
The internal frequency cancels under each substitution, as in
\eqref{eq:first-quintic-example}, so the resulting sums have the form
\begin{equation}\label{eq:vector-field-monomial}
 \begin{split}
 M_k(t,z)
 &=\sum_{k_1-k_2+\cdots-k_{2q}+k_{2q+1}=k}
 c(k_1,\ldots,k_{2q+1};k)e^{-i\Omega t}
 z_{k_1}\overline{z_{k_2}}\cdots\overline{z_{k_{2q}}}z_{k_{2q+1}},\\
 \Omega&=\omega_{k_1}-\omega_{k_2}+\cdots-\omega_{k_{2q}}+
          \omega_{k_{2q+1}}-\omega_k.
 \end{split}
\end{equation}
The multiplier $c$ includes earlier layer restrictions and phase
denominators and is independent of $t$.  A bounded multiplier means
that $|c|\le C$ uniformly over all Fourier
indices; the bound may depend on the fixed order, the depth $K$,
and the chosen regions, but not on the Fourier indices.
All state derivatives are real derivatives: in a conjugated slot, a
direction $h$ enters as $\bar h$.  The degree and boundedness of
$H_{q-1}$ are verified in the construction below.

Insert $I=\sum_{j=0}^K P_j$ in each input of a multilinear sum $M$
of degree $d=2q+1$.  Selecting the blocks to be eliminated defines
\begin{equation}\label{eq:vector-cost-retention}
 (\mathcal Q_{K,q}M)(t;h_1,\ldots,h_d)
 =\sum_{\substack{0\le j_i\le K\\q+j_1+\cdots+j_d\le K}}
 M(t;P_{j_1}h_1,\ldots,P_{j_d}h_d).
\end{equation}
Here $j_i$ is the layer label of the $i$th input, and $P_K$ projects
outside $B_{K-1}$.  Since $q\ge1$, no selected block has an input
from $P_K$.  The complementary nonresonant blocks are retained in the
equation and estimated in Section~\ref{sec:transfer}.  Related use of
initial-mode information to select frequency combinations appears in
\cite[Section~3.2]{GaucklerHairerLubichWeiss2012}.

As before, $\mathcal R$ selects $\Omega=0$ and $\mathcal N$ selects
$\Omega\ne0$.  We cancel the unretained part of $H_{q-1}$ by solving
$\partial_t G_q=\mathcal Q_{K,q}\mathcal NH_{q-1}$.
Integration of each oscillatory factor gives
\begin{equation}\label{eq:vector-homological-inverse}
 c e^{-i\Omega t}\longmapsto
 \frac{c e^{-i\Omega t}}{-i\Omega}\quad(\Omega\ne0).
\end{equation}
Denote this operation by $\mathcal S$, with zero value on the resonant
part.  Then $G_q=\mathcal S\mathcal Q_{K,q}\mathcal NH_{q-1}$, and
the coordinate correction is $\delta^qG_q$.  This is the usual
phasewise solution of the homological equation
\cite{BambusiGrebert2006,Grebert2007}, restricted by the input filter.
In particular,
\[
 H_{q-1}-\partial_t G_q
 =\mathcal RH_{q-1}+(I-\mathcal Q_{K,q})\mathcal NH_{q-1}.
\]
Write these surviving parts as $Z_{q-1}=\mathcal RH_{q-1}$ and
$E_{q-1,K}=(I-\mathcal Q_{K,q})\mathcal NH_{q-1}$.

Now make the substitution
\begin{equation}\label{eq:additive-recursion}
 v^{[q-1]}=\phi_q(t,v^{[q]};\delta),\qquad
 \phi_q(t,z;\delta)=z+\delta^qG_q(t,z).
\end{equation}
Differentiation along the orbit gives
\[
 (I+\delta^qDG_q)\partial_t v^{[q]}
 =\delta X^{[q-1]}_\delta(t,v^{[q]}+\delta^qG_q)
  -\delta^q\partial_t G_q,
\]
where $G_q$ and its derivatives are evaluated at $(t,v^{[q]})$.
Denote the operator on the left by
$A_q(t,z;\delta)=I+\delta^qDG_q(t,z)$.
Applying its inverse to both sides and factoring out $\delta$ gives the exact new equation
$\partial_t v^{[q]}=\delta X^{[q]}_\delta(t,v^{[q]})$, where
\begin{equation}\label{eq:exact-pullback-recursion}
 X^{[q]}_\delta(t,z)
 =A_q^{-1}\left[
 X^{[q-1]}_\delta(t,z+\delta^qG_q(t,z))
 -\delta^{q-1}\partial_t G_q(t,z)\right].
\end{equation}

To express this field in the form needed for the next step, first
expand its shifted state argument at fixed $t,\delta$:
\begin{equation}\label{eq:state-shift-integral}
 \begin{split}
 &X^{[q-1]}_\delta(t,z+\delta^qG_q)
       -X^{[q-1]}_\delta(t,z)\\
 &\qquad=\delta^q\int_0^1
 D_zX^{[q-1]}_\delta(t,z+v\delta^qG_q)[G_q]\,\mathrm dv
 =:\delta^q\mathcal D_q(t,z;\delta).
 \end{split}
\end{equation}
Here $G_q=G_q(t,z)$ throughout the integral.  This expansion
applies to the entire exact old field, including its remainder.
We can now substitute \eqref{eq:step-taylor} at the unshifted argument
$(t,z)$ into \eqref{eq:exact-pullback-recursion}.  Using
$H_{q-1}-\partial_t G_q=Z_{q-1}+E_{q-1,K}$ gives
\[
 X^{[q]}_\delta
 =A_q^{-1}\left[
 \sum_{r=0}^{q-1}\delta^r(Z_r+E_{r,K})
 +\delta^q(\mathcal T_{q-1}+\mathcal D_q)\right].
\]
All fields in the sum are evaluated at $(t,z)$, and the
remainders at $(t,z;\delta)$.
Next use the exact identity $A_q^{-1}=I-\delta^qA_q^{-1}DG_q$ on
the first term in brackets.  This yields
\begin{equation}\label{eq:completed-step-field}
 \begin{split}
 X^{[q]}_\delta
 &=\sum_{r=0}^{q-1}\delta^r(Z_r+E_{r,K})\\
 &\quad+\delta^q A_q^{-1}\left[
 \mathcal T_{q-1}+\mathcal D_q
 -DG_q\left[\sum_{r=0}^{q-1}\delta^r(Z_r+E_{r,K})\right]\right].
 \end{split}
\end{equation}
Thus all lower coefficients have the required form.  Denote the
remaining factor multiplying $\delta^q$ by
\begin{equation}\label{eq:one-step-remainder}
 \mathcal U_q
 :=A_q^{-1}\left[
 \mathcal T_{q-1}+\mathcal D_q
 -DG_q\left[\sum_{r=0}^{q-1}\delta^r(Z_r+E_{r,K})\right]\right].
\end{equation}
To separate its leading coefficient from its remainder, Taylor's
formula in $\delta$ gives
\[
 \mathcal U_q(t,z;\delta)
 =\mathcal U_q(t,z;0)
  +\delta\int_0^1
    \left.\partial_\rho\mathcal U_q(t,z;\rho)
    \right|_{\rho=v\delta}\,\mathrm dv.
\]
For $q<K$, define the coefficient and integral on the right as
\begin{equation}\label{eq:next-coefficient-remainder}
 H_q(t,z)=\mathcal U_q(t,z;0),\qquad
 \mathcal T_q(t,z;\delta)
 =\int_0^1\left.\partial_\rho\mathcal U_q(t,z;\rho)
              \right|_{\rho=v\delta}\,\mathrm dv.
\end{equation}
We have obtained the exact expression
\begin{equation}\label{eq:next-step-taylor}
 X^{[q]}_\delta(t,z)
 =\sum_{r=0}^{q-1}\delta^r\bigl(Z_r(z)+E_{r,K}(t,z)\bigr)
  +\delta^qH_q(t,z)
  +\delta^{q+1}\mathcal T_q(t,z;\delta).
\end{equation}
This is \eqref{eq:step-taylor} with $q$ replaced by $q+1$.
For $q=1$, $\mathcal D_1(t,z;0)=DF(t,z)[G_1]$, so
$H_1=DF[G_1]-DG_1[C_0]$, agreeing with
\eqref{eq:first-quintic-coefficient}.  At the last step $q=K$, stop
at \eqref{eq:completed-step-field} and put $R_K=\mathcal U_K$.

The integral formulas also specify the remainders for subsequent
estimates.  On the common inner domains constructed below, they are
smooth in the state and $\delta$ and, in particular, satisfy
\begin{equation}\label{eq:step-remainder-bounds}
 \|\mathcal U_q\|_{h^s}+\|\mathcal T_q\|_{h^s}\le C,\qquad
 \|D_z\mathcal U_q[h]\|_{h^1}
  +\|D_z\mathcal T_q[h]\|_{h^1}\le C\|h\|_{h^1}
 \quad(q<K).
\end{equation}
The same bounds hold for $\mathcal U_K$.
For the finite Taylor expansions, we also retain $\delta$--state
derivative bounds through total order $K-q+1$ for $\mathcal U_q$
and $K-q$ for $\mathcal T_q$, with at most one state increment
in $h^1$ and the others in $h^s$.
The constants are uniform in $t$ and sufficiently small $\delta$.
These bounds follow by induction from the following convolution and
phase estimates and the bounded inverses in
Section~\ref{subsec:fixed-depth-construction}.

The convolution estimates follow the coefficientwise approach of
the tame-modulus method \cite{BambusiGrebert2006}.  Full proofs of
both lemmas are in Appendix~\ref{app:convolution-phase}.

\begin{lemma}[Convolution estimates]\label{lem:majorant-calculus}
Fix $s>2$, $d\ge1$, and signs $\sigma_i\in\{1,-1\}$.  For a
multiplier satisfying $|c(k_1,\ldots,k_d;k)|\le A$, define on
finitely supported inputs
\[
 M_k(h_1,\ldots,h_d)
 =\sum_{\sum_i\sigma_i k_i=k}c(k_1,\ldots,k_d;k)
   \prod_{\sigma_i=1}(h_i)_{k_i}
   \prod_{\sigma_i=-1}\overline{(h_i)_{k_i}}.
\]
The sum extends to inputs in $h^s$, and
\begin{align}\label{eq:majorant-Hs}
 \|M(h_1,\ldots,h_d)\|_{h^s}
 &\le C_{d,s}A\prod_i\|h_i\|_{h^s},\\
 \|M(h_1,\ldots,h_d)\|_{h^1}
 &\le C_{d,s}A\|h_\ell\|_{h^1}
       \prod_{i\ne\ell}\|h_i\|_{h^s}\quad(1\le\ell\le d).
 \label{eq:majorant-H1}
\end{align}
The second estimate also defines the sum when only the designated
input $h_\ell$ belongs to $h^1$.  Restricting the summation to any
set of Fourier tuples preserves these bounds.  Finite sums, real
state derivatives of $M(z,\ldots,z)$, and substitutions into their
input slots admit finite sums of estimates of the same kind.
\end{lemma}

\Needspace{12\baselineskip}
\begin{lemma}[Filtered phase gap]\label{prop:retained-gap}
Fix $K\ge2$ and one choice of the regions $B_j$ in
\eqref{eq:box-bands}.  For $1\le q\le K$, let
$0\le j_i<K$, $k_i\in B_{j_i}$ $(1\le i\le2q+1)$, and suppose
$q+\sum_i j_i\le K$.  Set
\[
 k=k_1-k_2+\cdots-k_{2q}+k_{2q+1},\qquad
 \Omega=\omega_{k_1}-\omega_{k_2}+\cdots-\omega_{k_{2q}}+
          \omega_{k_{2q+1}}-\omega_k.
\]
There is $D_K<\infty$, depending only on $K$ and the chosen regions,
such that every such phase with $\Omega\ne0$ satisfies
$|\Omega|^{-1}\le D_K$.  This applies to rectangles and to either
infinite strip, and hence to every phase divided by in
\eqref{eq:vector-homological-inverse} after filtering.
\end{lemma}

\Needspace{5\baselineskip}
\subsection{Composition and exact normal form}
\label{subsec:fixed-depth-construction}
We now justify all $K$ steps of \eqref{eq:current-transformed-equation}
on a common domain and identify the final equation.  With $b=v^{[K]}$,
the substitutions \eqref{eq:additive-recursion} give
\[
 a=(\phi_1\circ\cdots\circ\phi_K)(b)
   =:\Psi_{\delta,t}(b).
\]
\begin{theorem}[Exact partial normal form]\label{prop:partial-normal-form}
Fix $K\ge2$, $s>2$, $R>0$, and a choice of the regions $B_j$.
For sufficiently small $\delta_K>0$, the recursion
\eqref{eq:coefficient-definition}--\eqref{eq:exact-pullback-recursion}
defines all $K$ steps for $|\delta|\le\delta_K$.
On $U=B_{h^s}(0,R+1/4)$ the map $\Psi_{\delta,t}$ is a real
diffeomorphism onto its image, which contains
$\overline B_{h^s}(0,R)$.  It and its inverse are $C^1$ in $t$ and smooth in the state and
in $\delta$ at fixed $t$.  Their displacements are
$O(|\delta|)$.  Derivatives with respect to $\delta$ and the state through total order
$K+1$ have uniform bounds on inner domains, including bounds with
one state increment measured in $h^1$ and all others in $h^s$.
The bounds are uniform in $t$.

For $\delta>0$, setting $a=\Psi_{\delta,t}(b)$ gives exactly
\begin{equation}\label{eq:partial-normal-form}
 \partial_t b
 =\sum_{r=0}^{K-1}\delta^{r+1} Z_r(b)
  +\sum_{r=0}^{K-1}\delta^{r+1} E_{r,K}(t,b)
  +\delta^{K+1}R_K(t,b;\delta).
\end{equation}
Here $Z_r=\mathcal RH_r$, $Z_0=\mathcal RF$, and
$E_{r,K}=(I-\mathcal Q_{K,r+1})\mathcal NH_r$.
In each Fourier component these are homogeneous polynomial sums of
degree $2r+3$, of the form \eqref{eq:vector-field-monomial} with
$q=r+1$ and bounded multipliers.
The sum $Z_r$ contains only zero phases and is independent of $t$;
$E_{r,K}$ contains nonzero factors $e^{-i\Omega t}$ and only
the retained nonresonant input blocks.  If the $i$th input in such a block is
$P_{j_i}b$, its layer labels $j_i\in\{0,\ldots,K\}$ satisfy
\begin{equation}\label{eq:flat-remainder-form}
 r+j_1+\cdots+j_{2r+3}\ge K.
\end{equation}
The exact remainder $R_K$ is uniformly bounded in $h^s$ and
uniformly Lipschitz in $h^1$ on these balls.
Constants may depend on $K,R,s$ and the initial widths.
\end{theorem}

\begin{proof}
We argue by induction.

\emph{Induction hypotheses and initial step.}
For $0\le q\le K$, use the domains
$\mathcal O_q=B_{h^s}(0,R+1-q/(2K))$.
After $q<K$ steps we require the following properties on
$\mathcal O_q$ and a common interval of $\delta$ about zero.
At fixed $(t,z)$, denote the old field's $\delta$-Taylor coefficients by
\[
 M_m^{[q]}(t,z)
 :=\frac1{m!}\left.\partial_\delta^mX^{[q]}_\delta(t,z)
                          \right|_{\delta=0},\qquad 0\le m\le K-1.
\]
\begin{enumerate}[label=\textup{(\roman*)},leftmargin=2.4em]
\item The exact field $X^{[q]}_\delta$ has the form
\eqref{eq:next-step-taylor}, with $Z_r+E_{r,K}$ fixed for $r<q$.
Thus $M_m^{[q]}=Z_m+E_{m,K}$ for $m<q$, $M_q^{[q]}=H_q$,
and the coefficients with $m>q$ come from $\delta^{q+1}\mathcal T_q$.
\item The $\delta$ and state derivatives
$\partial_\delta^aD_z^bX^{[q]}_\delta$ for $a+b\le K+1$, and
$\partial_\delta^aD_z^b\mathcal T_q$ for $a+b\le K-q$, have
uniform $h^s$ bounds and the corresponding mixed bounds with one
state increment in $h^1$.  The fields are strongly continuous in $t$.
In particular, $\mathcal T_q$ is bounded in $h^s$ and uniformly
Lipschitz in $h^1$ on these balls.
\item The fixed coefficients with $m<q$ retain their Fourier structure.
For $q\le m\le K-1$, each $M_m^{[q]}$ is a finite sum of Fourier
convolutions of degree $2m+3$, with bounded multipliers and the
momentum and phase structure in \eqref{eq:vector-field-monomial}.
\end{enumerate}
The finite derivative ranges in \textup{(ii)} suffice for the
Taylor remainders and their mixed Lipschitz bounds used in
Lemma~\ref{lem:endpoint-transfer} and Section~\ref{sec:transfer}.
For $q=0$, $X^{[0]}_\delta=H_0=F$ and $\mathcal T_0=0$;
all three properties follow from the cubic convolution formula
and Lemma~\ref{lem:majorant-calculus}.

\emph{Passage from step $q$ to step $q+1$.}
The construction follows
$X^{[q]}_\delta\longrightarrow H_q\longrightarrow G_{q+1}
\longrightarrow X^{[q+1]}_\delta$.
By \textup{(iii)}, $H_q=M_q^{[q]}$ has degree $2q+3$.
Lemma~\ref{prop:retained-gap} bounds every phase reciprocal selected
by $\mathcal Q_{K,q+1}$.  Thus
\[
 G_{q+1}=\mathcal S\mathcal Q_{K,q+1}\mathcal NH_q,
 \qquad
 \partial_tG_{q+1}=\mathcal Q_{K,q+1}\mathcal NH_q
\]
have bounded multipliers.  The convolution estimates give their
required state-derivative bounds, including the mixed bounds;
the homological identity holds in $h^s$ and gives $G_{q+1}\in C^1_t$.
In particular, the polynomial structure supplies the needed
derivatives of $G_{q+1}$ independently of the finite derivative
range retained for $X^{[q]}_\delta$.

Decrease the interval of $\delta$ so that on $\mathcal O_{q+1}$,
\[
 |\delta|^{q+1}\|G_{q+1}\|_{h^s}<\frac1{4K},\qquad
 |\delta|^{q+1}\|DG_{q+1}\|_{\mathcal L(h^\nu,h^\nu)}<\frac12
 \quad(\nu=s,1).
\]
The shifted segment in \eqref{eq:state-shift-integral} then lies
in $\mathcal O_q$, and
\[
 A_{q+1}^{-1}
 =\sum_{\ell=0}^\infty(-\delta^{q+1}DG_{q+1})^\ell,
 \qquad
 \|A_{q+1}^{-1}\|_{\mathcal L(h^\nu,h^\nu)}\le2
 \quad(\nu=s,1).
\]
Consequently \eqref{eq:exact-pullback-recursion} defines the next
field.  Applying \eqref{eq:completed-step-field}--\eqref{eq:one-step-remainder}
with the step index $q+1$ gives the exact identity
\[
 X^{[q+1]}_\delta
 =\sum_{r=0}^{q}\delta^r(Z_r+E_{r,K})
       +\delta^{q+1}\mathcal U_{q+1},
\]
where $Z_q=\mathcal RH_q$ and
$E_{q,K}=(I-\mathcal Q_{K,q+1})\mathcal NH_q$.
Thus the coefficients of $\delta^r$ with $r<q$ are unchanged,
and the coefficient at order $q$ is now fixed as $Z_q+E_{q,K}$.
All these coefficients retain their Fourier structure.

We now verify regularity and uniform estimates for the new field
and remainder.  The exact recursion
\eqref{eq:exact-pullback-recursion} uses composition and a bounded
inverse, without differentiating the old field.  Standard finite-order
composition and inverse estimates therefore preserve its total
derivative order $K+1$ without a loss at successive steps.
The state integral \eqref{eq:state-shift-integral} uses one state
derivative of $X^{[q]}_\delta$, so $\mathcal D_{q+1}$ has bounds
through total order $K$.  Together with the $K-q$ orders available
for $\mathcal T_q$, formula~\eqref{eq:one-step-remainder} gives
$K-q$ orders for $\mathcal U_{q+1}$.
For the mixed estimate, differentiation of the state integral
puts the designated $h^1$ increment into
$(I+v\delta^{q+1}DG_{q+1})h$ or $DG_{q+1}[h]$, while
$G_{q+1}$ remains in $h^s$.
At higher orders the same increment occurs in only one branch of
each multilinear term.  On that branch use the inductive mixed
bounds for $X^{[q]}_\delta$, the convolution bounds for $G_{q+1}$,
and the $h^1$ bound for $A_{q+1}^{-1}$; estimate all other inputs
in $h^s$.
This proves the mixed bounds as well.  Strong time continuity is
preserved by these compositions and integrals.

It remains to verify the Fourier structure of the coefficients
at $\delta$ orders $q+1,\ldots,K-1$, which are needed in later steps.
The exact recursion \eqref{eq:exact-pullback-recursion} contains
\[
 X_\delta^{[q]}(t,z+\delta^{q+1}G_{q+1}).
\]
To organize its new $\delta$ coefficients, expand this composite term
in the state variable around $z$, and use the finite $\delta$-Taylor
expansion of the old field.  At $\delta$ order $r$, only $M_m^{[q]}$
with $m\le r$ can occur: inserting $j$ copies of $G_{q+1}$ gives
order $m+(q+1)j$.
Consequently, denoting the bracket coefficients by $N_r$, we obtain
for $0\le r\le K-1$
\[
 N_r=
 \sum_{\substack{m,j\ge0\\m+(q+1)j=r}}
 \frac1{j!}D_z^jM_m^{[q]}[G_{q+1},\ldots,G_{q+1}]
       -\mathbf1_{\{r=q\}}\partial_tG_{q+1}.
\]
The inverse series above contributes $(-DG_{q+1})^\ell$ at
$\delta$ order $(q+1)\ell$.  Consequently
\[
 M_r^{[q+1]}
 =\sum_{\ell=0}^{\lfloor r/(q+1)\rfloor}
       (-DG_{q+1})^\ell N_{r-(q+1)\ell}.
\]
Both sums are finite.  The uniform finite-order derivative bounds
give an $O_{h^s}(\delta^K)$ remainder when collecting coefficients
through order $K-1$; no infinite Taylor expansion is used.
For a term in $N_r$ with $m+(q+1)j=r$, its state degree is
\[
 (2m+3)-j+j(2q+3)=2r+3.
\]
The subtracted term has degree $2q+3$ at order $q$.
Each application of $DG_{q+1}$ increases the degree by $2(q+1)$,
so every term in $M_r^{[q+1]}$ has degree $2r+3$.

Momentum and phase are preserved by each slot substitution.
Indeed, replace the $a$th slot of an outer monomial, with index
$\kappa=k_a$ and sign $\sigma=\sigma_a$, by an inner monomial
with signs $\eta_j$ and momentum
$\kappa=\sum_j\eta_j\ell_j$.  The new signs are the unchanged
outer signs and $\sigma\eta_j$; their sum remains one.
The phase is
\[
 \begin{aligned}
 \Omega_{\rm out}+\sigma\Omega_{\rm in}
 &=\left(\sum_{i\ne a}\sigma_i\omega_{k_i}
             +\sigma\omega_\kappa-\omega_k\right)
   +\sigma\left(\sum_j\eta_j\omega_{\ell_j}-\omega_\kappa\right)\\
 &=\sum_{i\ne a}\sigma_i\omega_{k_i}
             +\sum_j\sigma\eta_j\omega_{\ell_j}-\omega_k.
 \end{aligned}
\]
Thus the internal frequency cancels for either sign, and the
new momentum relation follows by the same substitution.
At degree $2r+3$ the sign sum one gives $r+2$ unconjugated and
$r+1$ conjugated factors, which can be reordered as in
\eqref{eq:vector-field-monomial}.
The new multipliers are finite products of old multipliers or
their conjugates.  Internal indices are fixed by momentum, so
no additional free sum appears in these coefficients.
The earlier input restrictions also remain, preserving every
inherited phase-denominator bound.  This proves \textup{(iii)}
through order $K-1$.

If $q+1<K$, formula~\eqref{eq:next-coefficient-remainder} now
defines $H_{q+1}=\mathcal U_{q+1}(t,z;0)$ and $\mathcal T_{q+1}$,
and yields \eqref{eq:next-step-taylor} at the next step.
After extracting $H_{q+1}$, the integral formula for $\mathcal T_{q+1}$
uses one $\delta$ derivative of $\mathcal U_{q+1}$, leaving $K-q-1$
total derivative orders for this remainder and proving \textup{(ii)} in its required range.
Moreover, $H_{q+1}=M_{q+1}^{[q+1]}$ has exactly the structure
needed to construct $G_{q+2}$.  Hence all three induction properties
pass from $q$ to $q+1$.

\emph{Completion of the equation and composite transformation.}
At the last step set $R_K=\mathcal U_K$.
The completed-step identity, multiplied by the common factor
$\delta$ in the physical-time equation, gives
\eqref{eq:partial-normal-form}.
The remaining derivative order is at least one, so $R_K$ is
uniformly bounded in $h^s$ and uniformly Lipschitz in $h^1$.
Equivalently, Taylor's formula at fixed $(t,b)$ gives
\begin{equation}\label{eq:parameter-integral-remainder}
 R_K(t,b;\delta)=\frac1{(K-1)!}
 \int_0^1(1-v)^{K-1}
 \left.\partial_\rho^K X^{[K]}_\rho(t,b)
 \right|_{\rho=v\delta}\,\mathrm dv.
\end{equation}
The $K+1$ derivative orders for $X^{[K]}_\delta$ justify this
formula and its mixed Lipschitz bound.
The resonant $Z_r$ is independent of $t$.
A retained block fails $(r+1)+\sum_i j_i\le K$, which, for
integer indices, is precisely \eqref{eq:flat-remainder-form}.

There are only $K$ steps, so take the minimum of their thresholds for $\delta$.  For the coordinate maps, decrease it further so that
on the common ball $\mathcal O_K=B_{h^s}(0,R+1/2)$ the bounds
\[
 \begin{aligned}
 d_q&=\sup_{\substack{t\in\mathbb R,\ z\in\mathcal O_K\\|\delta|\le\delta_K}}
                  |\delta|^q\|G_q(t,z)\|_{h^s},\\
 L_q&=\max_{\nu=s,1}\ 
       \sup_{\substack{t\in\mathbb R,\ z\in\mathcal O_K\\|\delta|\le\delta_K}}
           |\delta|^q\|DG_q(t,z)\|_{\mathcal L(h^\nu,h^\nu)}
 \end{aligned}
\]
satisfy $\sum_{q=1}^Kd_q<1/8$ and
$\sum_{q=1}^K L_q<1/8$.
Integrating the derivative bound along the segment from $z'$ to $z$
shows that the correction $\delta^qG_q$ is $L_q$-Lipschitz on
$\mathcal O_K$.  The reverse triangle inequality then gives
\[
 \|\phi_q(t,z;\delta)-\phi_q(t,z';\delta)\|_{h^s}
       \ge(1-L_q)\|z-z'\|_{h^s}
\]
and hence injectivity on $\mathcal O_K$.
The forward composition from $U$ stays in $\mathcal O_K$, since
its total displacement is at most $\sum_qd_q$.

To prove image coverage, start from $a\in\overline B_{h^s}(0,R)$
and invert the steps in the order $1,\ldots,K$.
At step $q$ the preceding value $w$ satisfies
$\|w\|_{h^s}\le R+\sum_{j<q}d_j$.
The map $z\mapsto w-\delta^qG_q(t,z)$ is a contraction of
$\overline B_{h^s}(w,d_q)\subset\mathcal O_K$ into itself.
Its fixed point is the required preimage and moves by at most
$d_q$.  Hence the final inverse point belongs to $U$,
proving $\overline B_{h^s}(0,R)\subset\Psi_{\delta,t}(U)$.
The actual total displacement is $O(|\delta|)$ in either direction.
The bounded inverses $A_q^{-1}$ and the standard inverse function
and composition theorems give a smooth diffeomorphism, its
$C^1$ time dependence, and the $\delta$--state derivative bounds through
total order $K+1$, including the mixed bounds.
Finally, the chain-rule identities in Section~\ref{subsec:recursive-construction}
are reversible at each step.  They give exact equivalence of the
equations, with initial data
$b(0)=\Psi_{\delta,0}^{-1}(a(0))$.
\end{proof}

\subsection{Fourier support of the normal form}
To use \eqref{eq:partial-normal-form} for layer estimates, we need the
support of its resonant coefficients $Z_r$ and of the coordinate map
$\Psi_{\delta,t}$.  We relate the polynomial degree of these
coefficients to their support growth.
For the chosen regions set $\widetilde B_j^K=B_j$ for $j<K$ and
$\widetilde B_K^K=\mathbb Z^2$.  We say that a multilinear map has
filtration order $p$ if inputs supported in
$\widetilde B_{j_1}^K,\ldots,\widetilde B_{j_d}^K$ have output in
$\widetilde B_{\min\{K,p+\sum_i j_i\}}^K$.
When an input index equals $K$, this condition imposes no restriction.

The next observation is the coordinatewise resonance structure used
in \cite[Section~4]{StaffilaniWilson2020}.
\begin{lemma}[Coordinatewise cubic selection]\label{lem:cubic-selection}
Let $k_\ell=(m_\ell,n_\ell)$ and $k=(m,n)=k_1-k_2+k_3$.
For the cubic phase
\[
 \Omega=\omega_{k_1}-\omega_{k_2}+\omega_{k_3}-\omega_k,
\]
one has
\begin{equation}
 -\frac{\Omega(k_1,k_2,k_3,k)}8
 =(m_1-m_2)(m_3-m_2)
 +\pi^2(n_1-n_2)(n_3-n_2).
 \label{eq:cubic-phase-factorization}
\end{equation}
Consequently, $\Omega=0$ if and only if
\[
 (m_1-m_2)(m_3-m_2)=0,\qquad
 (n_1-n_2)(n_3-n_2)=0.
\]
Each output coordinate is therefore copied from one of the positive-sign
inputs $k_1,k_3$, possibly from different inputs in the two coordinates.
\end{lemma}
\begin{proof}
Expanding the squares in the definition of $\Omega$ gives
\eqref{eq:cubic-phase-factorization}.  The two products are integers,
so irrationality of $\pi^2$ forces both to vanish when $\Omega=0$;
the converse is immediate.  If $m_1=m_2$, momentum gives $m=m_3$;
if $m_3=m_2$, it gives $m=m_1$.  The second coordinate is identical.
\end{proof}

For a coefficient of degree $2q+1$, the momentum relation is
\begin{equation}
 k=k_1-k_2+k_3-\cdots-k_{2q}+k_{2q+1}.
 \label{eq:higher-momentum}
\end{equation}
By \eqref{eq:vector-field-monomial}, zero total phase means
\begin{equation}
 \omega_k+\sum_{\ell\ {\rm even}}\omega_{k_\ell}
 =\sum_{\ell\ {\rm odd}}\omega_{k_\ell}.
 \label{eq:higher-energy}
\end{equation}
Irrationality of $\pi^2$ separates this equality into a sums-of-squares
identity in each coordinate.

\begin{proposition}[Filtration bounds]\label{prop:filtration-bounds}
Suppose the first input coordinates satisfy
$|m_\ell|\le(2j_\ell+1)M_0$, $j_\ell\ge0$.
\begin{enumerate}
\item Every degree-$2q+1$ momentum monomial satisfies
\begin{equation}
 |m|\le\bigl(2(q+j_1+\cdots+j_{2q+1})+1\bigr)M_0.
 \label{eq:general-filter-bound}
\end{equation}
Thus a coordinate-map coefficient multiplying $\delta^q$ has
filtration order $q$.
\item If \eqref{eq:higher-energy} holds and $q\ge2$, then
\begin{equation}
 |m|\le\bigl(2(q-1+j_1+\cdots+j_{2q+1})+1\bigr)M_0.
 \label{eq:resonant-filter-bound}
\end{equation}
Thus the resonant coefficient of degree $2q+1$, which multiplies
$\delta^q$ in the physical-time equation, has filtration order $q-1$.
\item A resonant cubic coefficient has filtration order zero.
\end{enumerate}
The corresponding bounds for the second coordinate follow by replacing
$m_\ell,m,M_0$ with $n_\ell,n,N_0$.  Apply both coordinate bounds
for rectangles and only the bounded-coordinate estimate for either strip.
\end{proposition}
\begin{proof}
Momentum gives
$|m|\le\sum_\ell|m_\ell|
\le(2\sum_\ell j_\ell+2q+1)M_0$, proving the first estimate.

For a resonance, the first-coordinate energy identity is
\[
 m^2+\sum_{\ell\ {\rm even}}m_\ell^2
 =\sum_{\ell\ {\rm odd}}m_\ell^2.
\]
Put $S=\sum_{\ell\ {\rm odd}}j_\ell$.  Discarding the nonnegative
input squares on the left and substituting the input coordinate bounds gives
\[
 \begin{aligned}
 m^2
 &\le\sum_{\ell\ {\rm odd}}m_\ell^2
 \le M_0^2\sum_{\ell\ {\rm odd}}(2j_\ell+1)^2\\
 &=M_0^2\left(4\sum_{\ell\ {\rm odd}}j_\ell^2+4S+q+1\right)
 \le M_0^2(4S^2+4S+q+1),
 \end{aligned}
\]
where the last inequality uses
$\sum_{\ell\ {\rm odd}}j_\ell^2\le S^2$.
The assumption $q\ge2$ ensures
\[
 (2S+2q-1)^2-(4S^2+4S+q+1)
 =8S(q-1)+q(4q-5)\ge0.
\]
Taking square roots and using $S\le\sum_\ell j_\ell$, we obtain
$|m|\le(2S+2q-1)M_0
\le\bigl(2(q-1+\sum_\ell j_\ell)+1\bigr)M_0$,
which proves the second estimate.
For $q=1$, coordinatewise cubic selection gives
$|m|\le\max\{|m_1|,|m_3|\}
\le(2\sum_\ell j_\ell+1)M_0$, proving the third.
Replacing $m_\ell,m,M_0$ with $n_\ell,n,N_0$ proves the
second-coordinate bounds; no estimate is needed in the unbounded
direction of a strip.
\end{proof}

\begin{lemma}[Endpoint coefficients]\label{lem:endpoint-transfer}
For $1\le q<K$, the coefficient of $\delta^q$ in
$\Psi_{\delta,t}-I$ or $\Psi_{\delta,t}^{-1}-I$
is homogeneous of degree $2q+1$ and has filtration order $q$.
Both maps have exact Taylor remainders equal to $\delta^K$ times
uniformly bounded maps, with uniform mixed Lipschitz bounds on
inner domains.
\end{lemma}
\begin{proof}
In the finite forward composition, every substitution adds twice its
$\delta$ order to the degree and preserves momentum.
Hence the coefficient of order $q$ has degree $2q+1$; its support
follows from \eqref{eq:general-filter-bound}.
For the inverse, differentiate
$\Psi_{\delta,t}(\Psi_{\delta,t}^{-1}(z))=z$
at $\delta=0$, with $t,z$ fixed.  At each order the new inverse
coefficient enters with identity coefficient, and all remaining
terms involve coefficients already determined.  Their finite
substitutions preserve degree, momentum, and bounded convolution
majorants.  The integral Taylor formula and the $\delta$--state derivative
bounds in Theorem~\ref{prop:partial-normal-form} give both exact
remainders and their mixed Lipschitz estimates.
\end{proof}

\section{Transfer to the exact solution}\label{sec:transfer}
We transfer the normal-form bounds to the exact solution on
$0\le t\le T\delta^{-1}$.
Starting from the leading resonant flow, we construct a finite
profile whose layer sizes, together with the filter condition, allow
the retained nonresonant contribution to be absorbed into the remainder
estimate.  An $h^1$ stability estimate then compares this profile with
the exact trajectory.

\subsection{The cubic resonant flow}
For the profile coefficients, let $\Phi_0(\sigma,\sigma_0;z)$ be the flow of
$\partial_\sigma z=Z_0(z)$, where $Z_0=\mathcal RF$.
The preservation of Fourier boxes by cubic resonances underlies
\cite[Sections~4--5]{StaffilaniWilson2020}.  We also need the
corresponding support property of the linearized flow.
\begin{lemma}[Principal flow and linearized support]\label{lem:principal-flow}
The flow $\Phi_0$ exists in $h^s$ on every bounded interval of $\sigma$.
Every row mass $\sum_n|z_{m,n}|^2$ and column mass
$\sum_m|z_{m,n}|^2$ is constant, and
$\|\Phi_0(\sigma,0;z)\|_{h^s}\le C_s\|z\|_{h^s}$.
If the initial value is supported in $B_0$, the trajectory stays in $B_0$.
Along such a trajectory, the linear variational propagator preserves
$P_{B_j}h^s$ for every $j$ and is bounded on both $h^s$ and $h^1$ on
each fixed interval of $\sigma$.  These conclusions hold for all three choices
of $B_j$ in \eqref{eq:box-bands}.
\end{lemma}
\begin{proof}
For a resonant quartet $k_1-k_2+k_3-k_4=0$, coordinatewise momentum
and energy conservation give
\[
 \{(k_1)_\nu,(k_3)_\nu\}
 =\{(k_2)_\nu,(k_4)_\nu\},\qquad \nu=1,2.
\]
Fix a row $m$ and write $w(k)=\boldsymbol1_{\{k^{(1)}=m\}}$.
Differentiating its mass gives twice the real part of $-i\lambda$
times the quartic sum with weight $w(k_4)$.  Denote this sum by $S$.
Exchange of the negative slots, the multiset identity, and exchange
of the positive slots yield
\[
 S=\frac12\sum_{\rm res}(w(k_2)+w(k_4))
             z_{k_1}\bar z_{k_2}z_{k_3}\bar z_{k_4}
  =\frac12\sum_{\rm res}(w(k_1)+w(k_3))
             z_{k_1}\bar z_{k_2}z_{k_3}\bar z_{k_4}
  =\bar S.
\]
Thus the row mass is constant.  Column indicators give the same result.
The sums converge absolutely since $h^s\subset\ell^1$.  Summing these
nonnegative conserved masses with coordinate weights gives
\[
 \|z\|_{h^s}^2\asymp_s
 \sum_m(1+4m^2)^s\sum_n|z_{m,n}|^2
 +\sum_n(1+4\pi^2n^2)^s\sum_m|z_{m,n}|^2.
\]
This bound and local existence for the bounded cubic vector field give
global continuation in $h^s$.

Lemma~\ref{lem:cubic-selection} preserves $B_0$.  Along a trajectory
$z(\sigma)$ supported in $B_0$, each term of $DZ_0(z(\sigma))h$
has two inputs in $B_0$ and the remaining input lies
in $B_j$.  Each output coordinate is copied from a positive input, so
the output lies in $B_j$, including when the increment occupies a
conjugated slot.  Uniqueness for the linear equation proves support
invariance.  Its $h^s$ and $h^1$ bounds follow from
\eqref{eq:majorant-Hs}--\eqref{eq:majorant-H1} and Gronwall on the fixed
interval.  The argument uses only the bounded coordinates for a strip.
\end{proof}

\subsection{Finite graded expansion and stability}
Approximation by resonant amplitudes and stability comparison are used
in \cite[Section~2 and Appendix~A]{CarlesFaou2012}; finite modulation
expansions also underlie \cite{GaucklerLubich2010}.
We first construct a finite profile with the required layer structure,
then control its error relative to the exact solution of the final
normal-form equation throughout the target interval.  This yields the
Fourier energy strata of the exact solution.
\begin{lemma}[Transfer at a fixed depth]\label{lem:filtered-transfer}
Let $b$ solve \eqref{eq:partial-normal-form} on $0\le t\le T\delta^{-1}$,
with actual initial value
\[
 b(0)=\Psi_{\delta,0}^{-1}(\widehat u_0),\qquad
 \widehat u_0\in h^s.
\]
Assume either $\supp\widehat u_0\subset B_0$, for any of the three
support geometries, or the rectangular geometric-tail condition
\[
 \|P_{\Gamma_j}u_0\|_{H^1}^2\le E_*\eta^j\quad(j\ge0),
 \qquad E_*>0,\quad 0<\eta<1.
\]
In the first case set $y=0$;
in the second set $y=\sqrt\eta$.  Put
\[
 h=\max\{\delta,y\},\qquad E=\|P_K\widehat u_0\|_{h^1}.
\]
Assume that the normal-form construction uses a ball containing
$\widehat u_0$, the exact trajectory $b(t)$, and the leading trajectory
$\Phi_0(\sigma,0;P_{B_0}\widehat u_0)$, $0\le\sigma\le T$,
with a fixed positive boundary margin.
Then, for sufficiently small $h$,
\begin{equation}\label{eq:abstract-layer-bound}
 \begin{aligned}
 \sup_{0\le t\le T\delta^{-1}}\|P_jb(t)\|_{h^1}
 &\le C_K(h^j+E)\quad(1\le j<K),\\
 \sup_{0\le t\le T\delta^{-1}}\|P_Kb(t)\|_{h^1}
 &\le C_K(h^K+E).
 \end{aligned}
\end{equation}
The constants and the smallness threshold may depend on the fixed
construction bounds and, in the geometric-tail case, on $E_*$,
but are uniform in $\delta,y$ and $t$.
\end{lemma}
\begin{proof}
We first prepare initial values with the required layer structure
for the profile coefficients.  Decompose the original initial value as
\[
 \widehat u_0=\sum_{j<K}y^jv_j+P_K\widehat u_0.
\]
If $\supp\widehat u_0\subset B_0$, take $y=0$,
$v_0=\widehat u_0$, and $v_j=0$ for $1\le j<K$.
Then $E=0$ and $\|v_0\|_{h^s}=\|\widehat u_0\|_{h^s}$.
For a rectangular geometric tail, take $y=\sqrt\eta$ and
$E=\|P_K\widehat u_0\|_{h^1}$, and define the finite-layer coefficients by
\[
 P_{\Gamma_j}\widehat u_0=y^jv_j,\qquad
 \supp v_j\subset\Gamma_j,\qquad
 \|v_j\|_{h^s}\le C_KE_*^{1/2}\quad(0\le j<K).
\]
The norm bound follows from the bounded frequencies in these
rectangular layers.  Orthogonality and
\eqref{eq:geometric-initial-tail} give
\begin{equation}\label{eq:orthogonal-initial-tail}
 E\le\frac{E_*^{1/2}y^K}{\sqrt{1-y^2}}.
\end{equation}
To obtain the corresponding coefficients of the transformed initial
value, use an auxiliary variable $\rho$ to count the combined powers
of $\delta$ and $y$.  In both cases, expand the following map at
$\rho=0$ through order $K-1$, with the ratios $\delta/h,y/h$ fixed:
\[
 \rho\longmapsto
 \Psi_{\rho\delta/h,0}^{-1}
 \left(\sum_{j<K}\rho^j(y/h)^jv_j\right),\qquad 0\le\rho\le h.
\]
The input remains in the construction ball: each Fourier layer is
multiplied by a factor between zero and one, so its $h^s$ norm is at
most $\|\widehat u_0\|_{h^s}$.
By Lemma~\ref{lem:endpoint-transfer}, an inverse coefficient of
order $r$ has filtration order $r$;
after substitution, a term of total order $\ell=r+\sum_i j_i$
is therefore supported in $B_\ell$.
The finite Taylor formula and the mixed Lipschitz bound for the omitted
exterior give
\[
 b(0)=\sum_{\ell=0}^{K-1}h^\ell b_\ell^0+e^0,
 \qquad b_0^0=P_{B_0}\widehat u_0,
\]
\[
 \supp b_\ell^0\subset B_\ell,\qquad
 \|b_\ell^0\|_{h^s}\le C_K,\qquad
 \|e^0\|_{h^1}\le C_K(E+h^K).
\]
These bounds are uniform in $\delta,y$, since $\delta/h,y/h\le1$.

To construct the profile coefficients, use the local variable
$\tau=\delta t\in[0,T]$ and prescribe $\beta_\ell(0)=b_\ell^0$.
First solve $\partial_\tau\beta_0=Z_0(\beta_0)$.
For $1\le\ell<K$, the coefficients already constructed determine
\[
 \mathcal F_\ell=
 [\rho^\ell]\sum_{r=0}^{\ell}\rho^r(\delta/h)^r
 Z_r\!\left(\sum_{m=0}^{\ell-1}\rho^m\beta_m\right).
\]
Here $[\rho^\ell]$ extracts a coefficient from a finite polynomial.
The equation for the current unknown coefficient is the inhomogeneous
linear equation
\[
 \partial_\tau\beta_\ell
 =DZ_0(\beta_0)\beta_\ell+\mathcal F_\ell.
\]
In particular, all higher resonant terms belong to the known forcing
$\mathcal F_\ell$.  Every forcing monomial satisfies
$r+\sum_i m_i=\ell$, so Proposition~\ref{prop:filtration-bounds}
places its support in $B_\ell$.
Let $U_0(\tau,\sigma)$ be the propagator of
$\partial_\tau w=DZ_0(\beta_0(\tau))w$.
Lemma~\ref{lem:principal-flow} applies to this linear propagation,
and Duhamel's formula gives
\[
 \beta_\ell(\tau)=U_0(\tau,0)b_\ell^0
 +\int_0^\tau U_0(\tau,\sigma)\mathcal F_\ell(\sigma)\,d\sigma.
\]
The support preservation and boundedness of $U_0$, together with the
convolution estimates for the known forcing, yield inductively
\[
 \supp\beta_\ell\subset B_\ell,
 \qquad \sup_{0\le\tau\le T}\|\beta_\ell\|_{h^s}\le C_K
 \quad(0\le\ell<K).
\]
Consequently, for small $h$ the finite profile
\[
 b^{<K}(t)=\sum_{\ell<K}h^\ell\beta_\ell(\delta t)
\]
stays in the construction ball.  To see the cancellation at each order,
substitute this profile into the resonant polynomials and use
$\delta^r=h^r(\delta/h)^r$.  Collecting powers of $h$ defines
the finite coefficients $\mathcal C_n$:
\[
 \begin{aligned}
 \sum_{r<K}\delta^{r+1}Z_r(b^{<K}(t))
 &=\delta\sum_{r<K}h^r(\delta/h)^r
 Z_r\!\left(\sum_{\ell<K}h^\ell\beta_\ell(\delta t)\right)\\
 &=\delta\sum_{n=0}^{J}h^n\mathcal C_n(\delta t),
 \end{aligned}
\]
where $J$ depends only on $K$.
The coefficient equations are precisely
$\partial_\tau\beta_\ell=\mathcal C_\ell$ for $0\le\ell<K$.
Differentiating the profile in physical time and subtracting therefore gives
\[
 \partial_t b^{<K}
 -\sum_{r<K}\delta^{r+1}Z_r(b^{<K})
 =-\delta\sum_{n=K}^{J}h^n\mathcal C_n(\delta t)
 =O_{h^s}(\delta h^K).
\]
The last estimate follows from the uniform bounds for the $\beta_\ell$,
the inequality $\delta/h\le1$, and the convolution estimates.

To estimate the terms omitted from that equation, observe that
\[
 \|P_jb^{<K}\|_{h^s}\le C_Kh^j\quad(0\le j<K),
 \qquad P_Kb^{<K}=0.
\]
Each block of $\delta^{r+1}E_{r,K}(t,b^{<K})$ is therefore bounded
by $C_K\delta h^{r+\sum_i j_i}$, or vanishes if an input is in the outer block.
Condition~\eqref{eq:flat-remainder-form} and the finitely many block
choices give $O_{h^s}(\delta h^K)$.  The term $\delta^{K+1}R_K$ has the same
bound.  Thus the profile has residual $O_{h^1}(\delta h^K)$ in the full exact
equation, uniformly for $0\le t\le T\delta^{-1}$.

The mixed bounds give an $h^1$ Lipschitz constant $C_K\delta$ for the
physical-time vector field on the ball containing the exact and approximate
trajectories.  Gronwall yields
\[
 \|b(t)-b^{<K}(t)\|_{h^1}
 \le e^{C_K\delta t}\bigl(\|e^0\|_{h^1}+C_K\delta t\,h^K\bigr),
\]
and hence, since $\delta t\le T$,
\[
 \sup_{0\le t\le T\delta^{-1}}\|b(t)-b^{<K}(t)\|_{h^1}
 \le C_K(E+h^K).
\]
For $j<K$, projection by $P_j$ removes all coefficients of order below
$j$; $P_K$ removes the entire profile.  This proves
\eqref{eq:abstract-layer-bound}.
\end{proof}

\subsection{Proof of the main theorem}
\begin{proof}[Proof of Theorem~\ref{cor:fixed-order}]
Throughout this proof, constants and thresholds may depend on the fixed
data of the theorem, but are independent of $\delta,\eta$.
Dependence on $E_*$ is indicated when relevant.
Set $K=N+1$.  Lemma~\ref{lem:principal-flow} and the exact-solution
envelope provide a ball containing the required exact and leading
resonant trajectories with a fixed margin.
Apply Theorem~\ref{prop:partial-normal-form} on a slightly larger ball
and choose a common $\delta_N^*>0$ for the three support geometries.
For $0<\delta\le\min\{\bar\delta,\delta_N^*\}$, the inverse
map defines the bounded exact orbit
$b(t)=\Psi_{\delta,t}^{-1}(a(t))$ on $[0,T\delta^{-1}]$.

If $\supp\widehat u_0\subset B_0$, Lemma~\ref{lem:filtered-transfer},
with $y=E=0$, supplies the profile and an $O_{h^1}(\delta^K)$ error.
Set $a_{\rm app}(t)=\Psi_{\delta,t}(b^{<K}(t))$.
Since $a(t)=\Psi_{\delta,t}(b(t))$, the mixed Lipschitz bound transfers
the error from $b$ to $a$.
Expanding $a_{\rm app}$ at fixed $t$, Lemma~\ref{lem:endpoint-transfer}
and the profile's input-layer structure place its coefficient of total
order $\ell$ in $B_\ell$, with an $O_{h^1}(\delta^K)$ remainder
uniform in $t$.
Finally multiplication by $e^{-i\omega_kt}$ preserves the norm and
every Fourier projection.  This proves parts~\textup{(i)} and~\textup{(ii)}, after decreasing
$\delta_N^*$ if necessary.

For part~\textup{(iii)}, take $y=\sqrt\eta$ and $h=\max\{\delta,y\}$
as in Lemma~\ref{lem:filtered-transfer}.
That lemma and the forward endpoint expansion apply for $h\le h_0$,
where $0<h_0\le1/2$ may depend on $E_*$ but not on $\delta,y$.
They give layer bounds $C_{N,E_*}(\delta^j+y^j+E)$ and an exterior
bound $C_{N,E_*}(\delta^K+y^K+E)$.
For $j<K$, \eqref{eq:orthogonal-initial-tail} and $y\le h_0$ give
$E\le C E_*^{1/2}y^j$.  Since $E_*>0$ is fixed, the permitted constant
also absorbs the factor needed to replace $y^j$ by $E_*^{1/2}y^j$.
The exterior bound follows with the explicit factor
$(1-y^2)^{-1/2}$.  These are \eqref{eq:main-geometric-layer} and
\eqref{eq:main-geometric-tail} in the small-$h$ range.

For $h>h_0$, use the common envelope directly.  Indeed,
\[
 \delta^j+E_*^{1/2}y^j
 \ge\min\{1,E_*^{1/2}\}h_0^j,
\]
and the same lower bound holds for the exterior right-hand side with
$j=K$, since $(1-y^2)^{-1/2}\ge1$.
Increasing $C_{N,E_*}$ therefore covers this range by $R_s$.
Only this constant, and not $\delta_N^*$, absorbs the
$E_*$-dependent profile threshold.  All bounds are uniform in
$0<\eta<1$ under the common envelope.  The construction and estimates
use only $|\lambda|=1$, so the constants can be chosen for both signs.
\end{proof}

\section{Scope of the result}
\label{sec:long-time}
Theorem~\ref{cor:fixed-order} concerns $t\le T\delta^{-1}$;
increasing the fixed layer depth does not extend this interval.
Longer times require new dynamical estimates, as the following
example and stability calculation show.

\subsection{A quintic resonant term outside the initial rectangle}
\label{sec:quintic-channel}
The cubic resonant field $Z_0$ preserves $B_0$.  The first change of
variables gives
\begin{equation}\label{eq:first-vector-normal-form}
 H_1=DF[G_1]-DG_1[C_0],\qquad Z_1=\mathcal RH_1.
\end{equation}
The following example shows that this quintic resonant field need not
preserve $B_0$.

\begin{example}[A nonzero component outside $B_0$]\label{prop:quintic-channel}
Take $K\ge2$ and $(M_0,N_0)=(2,1)$, and set
\[
 a=(-2,-1),\quad b=(1,1),\quad c=(0,-1),\quad d=(-3,1),
 \qquad z=e_a+e_b+e_c,
\]
where $e_k$ is the unit Fourier sequence at $k$.
Then $\supp z\subset B_0$, $d\notin B_0$, and $z_d=0$, but
\begin{equation}\label{eq:quintic-vector-component}
 (Z_1(z))_d=-\frac{i\pi^2}{2(1+4\pi^2)}\ne0.
\end{equation}
\end{example}

All three inputs of each generator monomial used here lie in $B_0$,
so the filter in \eqref{eq:first-generator} includes it in $G_1$.
Moreover, $E_{0,K}(t,z)=0$, and the resonant projection of
$DG_1[Z_0]$ vanishes.  Thus \eqref{eq:quintic-vector-component}
follows by evaluating $\mathcal R(DF[G_1])$ at the stated input;
Appendix~\ref{app:quintic-calculation} gives the complete calculation,
including the multiplicities and phase cancellation.

This example proves that $Z_1$ does not preserve $B_0$.
It does not establish order-one energy transfer over longer times,
or sharpness of the layer exponents or the time interval in
Theorem~\ref{cor:fixed-order}.
\subsection{Limits of the present stability argument}
The comparison in Lemma~\ref{lem:filtered-transfer} has a Lipschitz
rate $C\delta$ and hence a Gronwall factor $e^{C\delta t}$.
This factor is uniform for $t\le T\delta^{-1}$, but at
$t=T\delta^{-L}$, $L\ge2$, it becomes
$e^{CT\delta^{1-L}}$.  Increasing the normal-form depth reduces the
algebraic remainder but does not by itself control this growth.

\begin{remark}[Phase shear]\label{rem:phase-shear}
Even the cubic resonant flow has state derivatives that grow on these
longer intervals.  On the single-mode state $Ae_k$, its physical-time
flow is
\[
 A\longmapsto A\exp(-i\lambda\delta|A|^2t).
\]
At $A=r>0$, differentiation in the radial direction gives
\[
 \left|\frac{\mathrm d}{\mathrm dr}
       \bigl(r e^{-i\lambda\delta r^2t}\bigr)\right|
 =\sqrt{1+4r^4\delta^2t^2}.
\]
At $t=T\delta^{-L}$ this equals
$\sqrt{1+4r^4T^2\delta^{2(1-L)}}$, which is unbounded as
$\delta\downarrow0$.  Thus conservation of the Sobolev norm does not
give uniform state-derivative bounds for this flow.
\end{remark}

These computations limit the present stability argument, not the
possible duration of Fourier layering; the single-mode orbit itself
has no energy transfer.  Longer stability intervals have been proved
under other hypotheses, near a plane wave with generic mass
\cite{FaouGaucklerLubich2013}, and for most sufficiently small,
high-regularity data on Diophantine non-rectangular flat tori
\cite{BernierCamps2025}.  Neither result supplies the missing
longer-time estimates for the initial-data class considered here.

\section*{Acknowledgements}
The author thanks Christian Lubich for helpful discussions.
The author acknowledges financial support from the China Scholarship
Council (CSC).

\appendix
\Needspace{5\baselineskip}
\section{The defocusing Sobolev envelope}\label{app:envelope}
\begin{proof}[Proof of Lemma~\ref{lem:defocusing-envelope}]
We adapt the classical Br\'ezis--Gallou\"et proof
\cite{BrezisGallouet1980} to the torus and retain $\delta$
throughout, so that the final bound depends on $\delta t$.
Local existence and continuation in $H^s$, $s>2$, follow from the
unitary Schr\"odinger group and the locally Lipschitz cubic map
\cite{Cazenave2003}; see also \cite[Proposition~3.8 and Remark~3.10]{Tao2006}
for the periodic local theory and continuation criterion.
Mass and defocusing energy are conserved on the maximal interval:
\[
 M(u)=\|u\|_{L^2}^2,\qquad
 E_\delta(u)=\tfrac12\|\nabla u\|_{L^2}^2
                  +\tfrac\delta4\|u\|_{L^4}^4.
\]
For $0<\delta\le1$, positivity and the initial energy bound give
\[
 \|u^\delta(t)\|_{H^1}^2
 \le \|u_0\|_{H^1}^2+\tfrac12\|u_0\|_{L^4}^4
 \le M^2,
\]
where $M\ge1$ is independent of $\delta$.

The periodic logarithmic estimate corresponding to
\cite[Lemma~2]{BrezisGallouet1980} follows by splitting the Fourier
series at $|k|=R$, $R\ge2$.  Cauchy--Schwarz and the two-dimensional
lattice sums give
\[
 \|v\|_{L^\infty}
 \le C\|v\|_{H^1}\sqrt{\log(2+R)}
      +C_s\|v\|_{H^s}R^{1-s}.
\]
The two sums used here are
$\sum_{|k|\le R}\langle k\rangle^{-2}\le C\log(2+R)$ and
$\sum_{|k|>R}\langle k\rangle^{-2s}\le C_sR^{2-2s}$.
Choosing $R=2+(\|v\|_{H^s}/M)^{1/(s-1)}$ yields, for
$\|v\|_{H^1}\le M$,
\[
 \|v\|_{L^\infty}^2
 \le C_sM^2\log(e+\|v\|_{H^s}).
\]
This Fourier argument applies equally to noninteger $s$.

The Moser product estimate (the periodic version of
\cite[Lemma~A.8]{Tao2006}),
$\||v|^2v\|_{H^s}\le C_s\|v\|_{L^\infty}^2\|v\|_{H^s}$,
combined with the unitary group in the Duhamel formula, now gives
\[
 D^+Y(t)\le C_s\delta M^2Y(t)\log Y(t),
 \qquad Y(t)=e+\|u^\delta(t)\|_{H^s}.
\]
Here $D^+$ is the upper right Dini derivative; the same inequality
follows by smooth approximation.  Integrating the inequality gives
\[
 Y(t)\le
 \exp\!\left(\log Y(0)\,e^{C_sM^2\delta t}\right).
\]
This bound excludes finite-time loss of $H^s$ regularity and hence
gives global continuation.  For $\delta t\le T$ it is independent
of $0<\delta\le1$, proving Assumption~\ref{ass:Hs-envelope} with
$\bar\delta=1$.
\end{proof}

\Needspace{5\baselineskip}
\section{Convolution and filtered phase estimates}\label{app:convolution-phase}
\begin{proof}[Proof of Lemma~\ref{lem:majorant-calculus}]
First take finitely supported inputs and put
$f_i(k)=|(h_i)_k|$ if $\sigma_i=1$ and
$f_i(k)=|(h_i)_{-k}|$ if $\sigma_i=-1$.
Taking absolute values in the defining sum gives
\[
 |M_k(h_1,\ldots,h_d)|\le A(f_1*\cdots*f_d)(k).
\]
Reflection preserves all the sequence norms used here because
$\langle-k\rangle=\langle k\rangle$.
For $k=\sum_i\sigma_i k_i$ and $r\ge0$,
\[
 \langle k\rangle^r\le C_{d,r}\sum_{i=1}^d\langle k_i\rangle^r.
\]
Applying this inequality inside the convolution and using Young's
inequality, with the weighted factor in $\ell^2$ and all others in
$\ell^1$, gives
\[
 \|M(h_1,\ldots,h_d)\|_{h^r}
 \le C_{d,r}A\sum_{i=1}^d
       \|h_i\|_{h^r}\prod_{j\ne i}\|h_j\|_{\ell^1}.
\]
Cauchy--Schwarz yields
\[
 \|h\|_{\ell^1}
 \le\left(\sum_{k\in\mathbb Z^2}\langle k\rangle^{-2s}\right)^{1/2}
        \|h\|_{h^s}.
\]
The lattice sum is finite for $s>1$, so taking $r=s$ proves
\eqref{eq:majorant-Hs}.

For \eqref{eq:majorant-H1}, fix the input $\ell$ to be measured in
$h^1$ and again distribute the single output weight among the inputs.
When it falls on $\ell$, put that weighted input in $\ell^2$ and
every other input in $\ell^1$.  When it falls on $i\ne\ell$, put
the unweighted $\ell$th input in $\ell^2$ and the weighted $i$th
input in $\ell^1$.  The additional estimate needed in this case is
\[
 \sum_k\langle k\rangle|h_k|
 \le\left(\sum_k\langle k\rangle^{-2(s-1)}\right)^{1/2}
       \|h\|_{h^s}\le C_s\|h\|_{h^s},
\]
which holds in two dimensions because $s>2$.  In both cases Young's
inequality gives the right-hand side of \eqref{eq:majorant-H1}.
The bounds extend the sums by density and multilinearity, including
the case with one $h^1$ input.  The absolute-value convolution above
also shows absolute convergence in each Fourier component.

A restriction on the Fourier tuples multiplies $c$ by an indicator
bounded by one, so it preserves both estimates.  A real state
derivative of $M(z,\ldots,z)$ is a sum over its input slots: a
direction $h$ replaces $z$ in an unconjugated slot and $\bar h$
replaces $\bar z$ in a conjugated slot.  Repeated derivatives give
finitely many such terms, to which the same estimates apply.

Finally, consider substitution of another convolution $N$ into a
slot of $M$, as in $DM(z)[N(z)]$.  For each slot choice the momentum
relation for $N$ uniquely determines its output index from its inputs.
After eliminating this index, the expanded expression is a single
convolution in the remaining inputs, with multiplier the product of
the two original multipliers, or the conjugate of one of them.
There is therefore no additional free Fourier summation in its
multiplier.  Its absolute value is bounded by the product of the
original bounds.  Summing the finitely many slot choices proves the
assertion for substitutions, and iteration proves it for all finite
derivative contractions in the recursion.  The density argument
justifies the same identities for general $h^s$ inputs.
\end{proof}

\begin{proof}[Proof of Lemma~\ref{prop:retained-gap}]
Write $k_i=(m_i,n_i)$ and $k=(m,n)$.  In a bounded coordinate,
denote its input values by $r_i$ and its base half-width by $R_0$
($M_0$ for the first coordinate or $N_0$ for the second).
The assumptions $k_i\in B_{j_i}$ and $q+\sum_i j_i\le K$ give
\[
 \sum_{i=1}^{2q+1}|r_i|
 \le\bigl(2\sum_i j_i+2q+1\bigr)R_0
 \le(2K+1)R_0.
\]
Momentum then bounds the corresponding output coordinate by the
same number.  For rectangular regions both coordinates of all inputs
and of the output are bounded.  Only finitely many phases occur;
the minimum absolute value of their nonzero members is positive.

For a strip bounded in $n$, introduce the integers
\[
 A=\sum_{i=1}^{2q+1}(-1)^{i+1}m_i^2-m^2,
 \qquad
 B=\sum_{i=1}^{2q+1}(-1)^{i+1}n_i^2-n^2.
\]
Then $\Omega=4(A+\pi^2B)$, while the preceding coordinate bound gives
\[
 |B|\le\sum_i n_i^2+n^2
 \le2(2K+1)^2N_0^2=:Q_K^y.
\]
If $B=0$ and $\Omega\ne0$, then $|A|\ge1$.  If $B\ne0$,
$|A+\pi^2B|\ge\dist(\pi^2B,\mathbb Z)$.
Consequently every nonzero phase satisfies
\[
 |\Omega|\ge
 4\min\left\{1,\min_{\substack{B\in\mathbb Z\\1\le|B|\le Q_K^y}}
                         \dist(\pi^2B,\mathbb Z)\right\}>0.
\]
The minimum is over finitely many positive numbers because $\pi^2$
is irrational, even though $A$ is unbounded.

For a strip bounded in $m$, the same argument bounds
$|A|\le2(2K+1)^2M_0^2=:Q_K^x$.  Writing
$\Omega=4\pi^2(B+A/\pi^2)$ gives
\[
 |\Omega|\ge
 4\pi^2\min\left\{1,\min_{\substack{A\in\mathbb Z\\1\le|A|\le Q_K^x}}
                         \dist(A/\pi^2,\mathbb Z)\right\}>0.
\]
Taking the reciprocal of the appropriate lower bound proves the
lemma.  Each denominator inherited from an earlier step keeps its
earlier bound, since its multiplier retains the corresponding input
restriction.  At fixed depth only finitely many such factors occur
in each coefficient.  Their products are bounded without requiring
a lower bound for unrestricted nonzero phases.
\end{proof}

\section{Calculation of the quintic example}\label{app:quintic-calculation}
We compute the full component $(Z_1(z))_d$ for the state in
Example~\ref{prop:quintic-channel}, starting from
\eqref{eq:first-vector-normal-form}.
Every nonzero cubic input of $G_1(t,z)$ has layer indices $(0,0,0)$,
so it is included by the filter for $K\ge2$.
The complementary sum vanishes: $E_{0,K}(t,z)=0$, hence
$C_0(t,z)=Z_0(z)$.  In $DG_1[Z_0]$, the nonzero phase of a
monomial of $G_1$ is combined with a zero phase from $Z_0$ and
therefore remains nonzero.  Consequently
\[
 (Z_1(z))_d=\bigl(\mathcal R(DF[G_1])(z)\bigr)_d.
\]
Put $g=G_1(t,z)$.  Differentiating the three input factors of $F$ gives
\[
 (DF(t,z)[g])_d=-i\lambda
 \sum_{p-q+r=d}e^{-i\Omega(p,q,r,d)t}
 \bigl(g_p\overline{z_q}z_r
       +z_p\overline{g_q}z_r
       +z_p\overline{z_q}g_r\bigr).
\]
With three positive and two negative inputs in $\{a,b,c\}$,
the momentum equation at $d$ forces their signed multiplicities
to be $(2,1,-2)$.  Thus the only possible quintic monomial is
$z_a^2z_b\overline{z_c}^{\,2}$, whose value at the stated $z$ is one.
Its total phase vanishes because
$2\omega_a+\omega_b=2\omega_c+\omega_d=36+12\pi^2$.

In the two nonconjugated slots of the displayed derivative,
the relevant intermediate indices are $2a-c=(-4,-1)$ and
$a+b-c=(-1,1)$.  In its conjugated slot, they are
$2c-a=(2,-1)$ and $2c-b=(-1,-3)$.
Formula~\eqref{eq:first-generator}, evaluated at $z$, gives
\[
 \begin{aligned}
 g_{(-4,-1)}&=\frac{\lambda}{-32}e^{32it},
 &g_{(-1,1)}&=\frac{2\lambda}{16}e^{-16it},\\
 g_{(2,-1)}&=\frac{\lambda}{-32}e^{32it},
 &g_{(-1,-3)}&=\frac{\lambda}{-8(1+4\pi^2)}
                         e^{8i(1+4\pi^2)t}.
 \end{aligned}
\]
The factor two in $g_{(-1,1)}$ counts the two orders of $a,b$.
Substituting these four components into the derivative sum gives
\[
 \begin{aligned}
 (DF(t,z)[g])_d=-i\lambda\bigl(&
 2e^{-32it}g_{(-4,-1)}+2e^{16it}g_{(-1,1)}\\
 &+2e^{32it}\overline{g_{(2,-1)}}
 +e^{8i(1+4\pi^2)t}\overline{g_{(-1,-3)}}\bigr).
 \end{aligned}
\]
The first two factors of two count the two nonconjugated slots;
the third counts the two orders of the outer inputs $a,b$.
Every exponential cancels its partner.  Since $\lambda^2=1$,
the full resonant component is
\[
 (Z_1(z))_d=-i\left(
 \frac{2}{-32}+\frac{4}{16}+\frac{2}{-32}
 -\frac{1}{8(1+4\pi^2)}\right)
 =-\frac{i\pi^2}{2(1+4\pi^2)}.
\]

\end{document}